\documentclass[12pt]{article}
\usepackage{amsmath}
\usepackage{amssymb}
\usepackage{mathrsfs}
\usepackage{amsthm}
\usepackage{authblk}
\usepackage{color}
\usepackage{graphicx}
\usepackage{caption}
\usepackage{subcaption}
\usepackage{soul}
\usepackage{bbm} 
\usepackage{ulem}
\usepackage{cancel} 
\usepackage{algorithm}
\usepackage{algorithmicx}
\usepackage{amssymb}
\theoremstyle{definition}
\newtheorem{definition}{Definition}
\newtheorem{theorem}{Theorem}
\newtheorem{lemma}{Lemma}
\newtheorem{corollary}{Corollary} 
\newtheorem{remark}{Remark}
\newtheorem{proposition}{Proposition}
\usepackage{hyperref}
\usepackage[margin=1.08in]{geometry}
\allowdisplaybreaks[4]

\newcommand{\be}{\begin{equation}}
	\newcommand{\ee}{\end{equation}}
\newcommand{\bea}{\begin{eqnarray*}}
	\newcommand{\eea}{\end{eqnarray*}}

\DeclareMathOperator*{\argmin}{arg\,min}

\newcommand{\x}{{\pmb x}}
\newcommand{\y}{{\pmb {y}}}
\newcommand{\z}{{\pmb z}}

\newcommand{\A}{{\pmb A}}

\newcommand{\tp}{\texttt{T}}
\newcommand{\mS}{\mathcal{S}}

\newcommand{\mA}{\pmb{A}}
\def\lV{\left\lVert}
\def\rV{\right\lVert}
\def\lv{\left\lvert}
\def\rv{\right\lvert}

\begin{document}
	\title{Adversarial Phase Retrieval via Nonlinear Least Absolute Deviation \footnote{This work was supported in part by the NSFC under grant numbers U21A20426, 12071426 and the National Key Research and Development Program of China under grant number 2021YFA1003500, the Project funded by China Postdoctoral Science Foundation 2023M733114.}}

\author{Gao Huang\footnote{ School of Mathematical Science, Zhejiang University, Hangzhou 310027, P. R. China, E-mail address: hgmath@zju.edu.cn}}	
\author{Song Li\footnote{ School of Mathematical Science, Zhejiang University, Hangzhou 310027, P. R. China, E-mail address: songli@zju.edu.cn}}	
\author{Hang Xu\footnote{ Corresponding Author, School of Physics, Zhejiang University, Hangzhou 310027, P. R. China, E-mail address: hangxu@zju.edu.cn}}
\affil{}
\renewcommand*{\Affilfont}{\small\it}
	
	\date{December 12, 2023}
	
	\maketitle
	\begin{abstract}
	We investigate the phase retrieval problem 
	perturbed by dense bounded noise and sparse outliers that can change an adversarially chosen $s$-fraction of the measurement vector. The adversarial sparse outliers may exhibit dependence on both the observation and the measurement.
	We demonstrate that the nonlinear least absolute deviation based on amplitude measurement can tolerate adversarial outliers
	at a fraction of $s^{*,1}\approx0.2043$, while the intensity-based model can tolerate a fraction of $s^{*,2}\approx0.1185$.
	Furthermore, we construct adaptive counterexamples to show that 
	the thresholds are theoretically sharp, thereby showing the presentation of phase transition in the adversarial phase retrieval problem when the corruption fraction exceeds the sharp thresholds. 
This implies that the amplitude-based model exhibits superior adversarial robustness in comparison with the intensity-based model. 
Corresponding experimental results are presented to further illustrate our theoretical findings.
	To the best of our knowledge, our results provide the first theoretical examination of the distinction in robustness performance between amplitude and intensity measurement.
	A crucial point of our analysis is that we explore the exact distribution of some combination of two non-independent Gaussian random variables and present the novel probability density functions to derive the sharp thresholds.
		\end{abstract}
		

\section{Introduction}
	Recovering a signal from the magnitude of its linear samples, commonly known as phase retrieval, has gained significant attention in various fields due to its wide range of applications,  including X-ray crystallography \cite{millane1990phase,miao1999extending}, astronomy \cite{fienup1987phase}, microscopy \cite{miao2008extending}, optics \cite{walther1963question} and diffraction imaging \cite{chai2010array}.
	Mathematically, the problem consists of recovering an unidentified discrete signal $\pmb{x}_0\in\mathbb{F}^{n}$ with $\mathbb{F} \in \{ \mathbb{R}, \mathbb{C}\}$ from the general measurements
\bea
b_i = \mathcal{A}_i\left(\x_0\right) + \eta_i,\quad i=1,\cdots,m.
\eea
Here, $\eta_i \in \mathbb{R}$ represents arbitrary measurement noise and the sample $\mathcal{A}_i(\x_0)$ may be of the two forms: $\vert\langle \pmb{a}_{i},\pmb{x}_0 \rangle\vert$ and $\vert\langle \pmb{a}_{i},\pmb{x}_0 \rangle\vert^{2}$, 
called amplitude and intensity measurement respectively with $\pmb{a}_{i} \in \mathbb{F}^{n}$ being the design measurement vectors known as a priori. \cite{shechtman2015phase} offers a comprehensive review of the theory and application of phase retrieval, serving as a valuable resource for readers seeking further insight into this subject matter.

A typical approach to recover  $\pmb{x}_0$ is to consider the least squares (LS) approach i.e. $\ell_2$-loss. 
Cand\`es et al. in \cite{candes2015phase2} developed a gradient-based method---Wirtinger flow on the objective related to the intensity measurements. Afterwards, some algorithms based on Wirtinger flow emerged, including \cite{chen2017solving,wang2017solving,zhang2017nonconvex}. 
\cite{candes2015phase2,chen2017solving} focused on intensity measurement, while
\cite{wang2017solving,zhang2017nonconvex} studied a variant of Wirtinger flow on the amplitude-based model and showed that the amplitude-based algorithms perform better in computational complexity and sample complexity than the intensity-based algorithms.

Since the $\ell_2$-loss relies heavily on the bounded variance of the noise vectors in most of the existing literature, 
the above methods cannot be directly applied when corrupted by heavy-tailed errors or outliers \cite{zhang2018median,duchi2019solving}.
To deal with this case, $\ell_1$-loss is invoked and this leads to the least absolute deviation (LAD) approach:
\bea
\min_{\pmb{x} \in\mathbb{F}^{n}} \| \mathcal{A}\left(\x_0\right) - \pmb{b} \|_1,
\eea
where $\mathcal{A}\left(\x_0\right)$ can be amplitude or intensity measurement.
In real-world applications, outliers may be very common in the measurements due to sensor failure and occlusions. 
And these failure and occlusions can often be adversarial.
Adversarial outliers\footnote{Adversarial outliers are the data points which can be arbitrarily large and may even depend on the measurement system $(\pmb{b}, \pmb{A})$ and ground-truth $\x_0$. For example, the noise $ \tilde{\pmb{\eta}}$ of a response $\pmb{y} = \mathcal{A}(\pmb{x}_0 + \pmb{\eta}) = \mathcal{A}(\pmb{x}_0) + \tilde{\pmb{\eta}}$ is highly dependent on the measurement.} is prevalent in many applications such as video surveillance \cite{li2004statistical}, face recognition \cite{de2003framework} and sensor calibration \cite{li2016low}. 
Dwork et al. \cite{dwork2007price} first explored the adversarial robust recovery in decoding and 
considered the conflict between the adversary, who corrupts a signal, and the decoder, who attempts to reconstruct the message.
There is also a series of work considering robust recovery in the case of measurements being wiped or even adversarially corrupted in various fields including compressed sensing \cite{chen2013robust,karmalkar2019compressed}, matrix completion \cite{li2013compressed}, robust PCA \cite{pimentel2017adversarial,diakonikolas2023nearly}, low-rank matrix recovery \cite{xu2022low}. In addition, the adversarial robustness has also been extensively researched in robust high-dimensional statistics during resent years  \cite{klivans2018efficient,diakonikolas2019efficient,diakonikolas2023nearly}.

However, phase retrieval under outliers or even adversarial outliers was studied in a minority of work. Most of them \cite{hand2017phaselift,li2016low,huang2022outlier} considered the intensity-based model and took use of PhaseLift \cite{balan2009painless,candes2013phaselift} to lift the signal vector $\x_0$ into rank-one matrix $\x_0\x_0^\tp$.
 \cite{yeh2015experimental} demonstrated that amplitude-based cost functions perform better than intensity-based ones by using simulations and experiments, especially when the dataset contains noise or model mis-match.
Furthermore, outliers arise frequently from the phase imaging applications \cite{weller2015undersampled,yeh2015experimental} due to various reasons such as detector failures, recording errors, and missing data or even model mis-match. Such mistakes can even be adversarial, thus considering the theoretical robustness performance of the phase retrieval model with adversarial outliers has great practical significance.

In this paper, we first consider the following amplitude measurements corrupted by dense bounded noise and adversarial sparse outliers:
	\begin{equation}\label{formu1}
		\pmb{b}=\lv\pmb{A}\x_{0}\rv+\pmb{\omega}+\pmb{z},\quad \|\pmb{z}\|_{0}\leq sm,
	\end{equation}
where $\A = [\pmb{a}_1, \dots,\pmb{a}_m]^\tp \in \mathbb{F}^{m\times n}$ is the sampling matrix, $|\pmb{A}\pmb{x}| = [\vert\langle \pmb{a}_{1},\pmb{x} \rangle\vert, \dots, \vert\langle \pmb{a}_{m},\pmb{x}_0 \rangle\vert]$. 
Here, $\pmb{\omega}$ denotes the $\ell_{1}$-bounded noise, $\pmb{z}$ represents the adversarial sparse outliers whose
	nonzero entries can be arbitrarily large and may even depend on the measurement system $(\pmb{b}, \pmb{A})$ and ground-truth $\x_0$, and $s$ is referred to as the adversarial corruption fraction. 
We call it \textbf{\emph{adversarial phase retrieval}} problem based on amplitude measurements.

We study the reconstruction property of the nonlinear least absolute deviation (nonlinear LAD) for amplitude measurement 
	\begin{eqnarray} \label{model1}
			&&\min   \lV\lv\A \x\rv-\pmb{b}\rV_{1}\\
			&&\text{s.t. } \pmb{x}\in \mathcal{K}.		 \nonumber
	\end{eqnarray}
Here, $\mathcal{K}$ is a general set to capture the structure of the signal and we assume throughout that $\x_{0}$ lies in it.
A general structural representation is to let $\mathcal{R}:\mathbb{R}^{n}\to\mathbb{R}$ be a regularization function that reflects some priori knowledge of $\x_{0}$. Then $\mathcal{K}$ can be set as $\mathcal{K}=\{\x\in\mathbb{R}^{n}:\mathcal{R}(\x)\le R\}$ where $R$ is a tuning parameter. 
For example, we can choose $\mathcal{R}(\x)=\lV\x\rV_{1}$ if $\x$ is sparse.
For Gaussian measurements, we show that the nonlinear robust outlier bound condition (nonlinear ROB condition, see Section \ref{robc}) 
holds for any submatrices consisting of a $(1- s)$-fraction of the measurements with high probability under appropriate samples, where $s < s^{*,1} \approx 0.2043$. 
Utilizing the nonlinear ROB condition, we can show the amplitude-based nonlinear LAD \eqref{model1} is additionally robust against a fraction $s^{*,1} \approx 0.2043$ of arbitrary errors.
To the best of our knowledge, it is the first sharp sparsity threshold of adversarial outliers in phase retrieval based on the amplitude-based model.

And \cite{wang2017solving,zhang2017nonconvex,soltanolkotabi2019structured} showed that the amplitude-based algorithms perform better in sample complexity and computational complexity experimentally and theoretically. 
\cite{zhang2018median} demonstrated that amplitude-based algorithms can tolerate a larger fraction of outliers than intensity-based algorithms by experiments but lack of theoretical guarantees.
Therefore, we are encouraged to explore and compare the robust performance of models based on the two kinds of measurements.

In the intensity-based nonlinear LAD model, we consider that
\begin{eqnarray} \label{model3}
			&&\min \lV\lv\mA \x\rv^{2}-\pmb{b}\rV_{1}\\
			&&\text{s.t. }   \pmb{x}\in \mathcal{K},\nonumber			 
	\end{eqnarray}
	with intensity measurements 
$\pmb{b}=\lv\pmb{A}\x_{0}\rv^2+\pmb{\omega}+\pmb{z}$ where $|\pmb{A}\pmb{x}|^2 = [\vert\langle \pmb{a}_{1},\pmb{x} \rangle\vert^2, \dots, \vert\langle \pmb{a}_{m},\pmb{x}_0 \rangle\vert^2]$.
We use the similar analytical approach to give the sharp sparsity threshold $s^{*,2}\approx 0.1185$ of adversarial outliers, 
which is consistent with the bound of \cite{huang2022outlier}.
Compared to \cite{huang2022outlier}, we consider the nonlinear model instead of the one using PhaseLift and remove the Rademacher distribution assumption of the outliers to give more general analysis.

Our results not only give the two sharp sparsity thresholds of adversarial outliers in the models based on amplitude and intensity measurements, but also theoretically analyze the robust performance of the two models and explain the superiority of the amplitude-based nonlinear LAD model (\ref{model1}) compared to the intensity-based nonlinear one (\ref{model3}), since $s^{*,1}$ is larger than $s^{*,2}$. 
We also construct adaptive counterexamples that are related to measurement $\mA$ and observation $\pmb{b}$ to show that the recovery fails when adversarial corruption fraction exceeds $s^{*,1}$ and $s^{*,2}$ in the two cases.  We conduct extensive numerical experiments to corroborate our results.

\subsection{Our Contributions}
We summarize the main contributions of this work.
We first give a general framework in phase retrieval for the analyses of nonlinear LAD model based on a general set $\mathcal{K}$ under the corruption of adversarial sparse outliers. It is an underlying paradigm for dealing with different measurements to elaborate the corresponding ROB conditions. 
We present the probability density function (PDF) of the combination $\bigl|\lv X \rv-\lv Y\rv\bigr|$ and $\lv X\cdot Y \rv$ of two Gaussian variables (may be not normal or independent) $X$, $Y$. The PDF of $\bigl|\lv X \rv-\lv Y\rv\bigr|$ may be first presented in this paper.
Based on this, we give the nonlinear ROB condition to cope with the adversarial outliers.
We then give the sharp sparsity threshold $s^{*,1} \approx 0.2043$ of adversarial outliers in the nonlinear LAD model for amplitude measurement.
For intensity-based nonlinear LAD model, we give the sharp sparsity threshold $s^{*,2} \approx 0.1185$ without the Rademacher distribution assumption for outliers.  
Our analyses of these two models leads to the first theoretical examination of the distinction in robustness performance between amplitude and intensity measurement such that the amplitude-based nonlinear LAD model has superior adversarial robustness in comparison to the intensity-based one.

%




	
	\subsection{Comparisons to Related Work}\label{related}

A majority of existing work has considered stable performance in phase retrieval \cite{candes2013phaselift,eldar2014phase,chen2017solving,zhang2017nonconvex,huang2024performance}.
These work aimed to recover the signal from observations with bounded noise or random noise.
Only a handful of papers analyzed the phase retrieval problem in the presence of outliers, especially adversarial outliers \cite{hand2017phaselift,li2016low,duchi2019solving}.
\cite{hand2017phaselift} proved that the PhaseLift with LAD can tolerate a small, fixed fraction of gross outliers and
\cite{li2016low} restricted the signs of sparse noise to be generated from Rademacher distribution and 
extended the performance guarantee to the general low-rank PSD matrix case. 
\cite{duchi2019solving} developed the prox-linear algorithm to solve phase retrieval problems even with adversarially faulty measurements. 
Most of them did not consider the adaptive outliers but set the outliers to obey some specified distribution such as Gaussian or uniform distribution in numerical simulation. \cite{huang2022outlier} provided corresponding threshold $s^{*}\approx0.1185$ for robust PhaseLift model, considering intensity measurement.
They also assumed  that the signs of adversarial sparse outliers follow the Rademacher distribution as \cite{li2016low}.
Our paper considers adversarial sparse outliers without any additional assumptions and we give the sharp theoretical thresholds of the adversarial corruption fraction $s$ for both amplitude and intensity measurement. 

Adversarial analysis could be at the intersection of two lines of research, signal recovery and privacy-preservation. 
In decoding, \cite{dwork2007price} first provided that adversary accounts for at most 0.239 in the measurement by statistical methods, and the result has been successively extended to compressed sensing \cite {karmalkar2019compressed} and low-rank matrix recovery \cite{xu2022low}. These work could be utilized in privacy-preservation, which is a very important and popular theme in the era of big data. For example, \cite{zhou2009compressed} showed that a similar setting can be used for sparse regression and analyzed its privacy properties. Besides, researchers have applied deep learning to solve the phase retrieval problem such as \cite{kappeler2017ptychnet,metzler2018prdeep}.
They used convolutional neural network to denoise and create new algorithms. Their algorithms are exceptionally robust to noise in simulation but lack of theoretical analysis. 
We hope that our work serves as an important step towards privacy-preserving phase retrieval and decoding by nonlinear encryption, and also give some theoretical support for these novel algorithms in the future.

	\subsection{Notations and Preliminaries}
	Throughout this paper, we use the following notations. We first define two different metrics.	
	For amplitude-based model 
	we use metric
	\begin{equation}\label{dist1}
		\mathrm{dist}_{1}\left(\x,\y\right)=\min\left\{\lV\x-\y \rV_2,\lV\x+\y \rV_2\right\}
	\end{equation}
	and for intensity-based model the distance between $\x$ and $\y$ can be defined as
		\begin{equation}\label{dist2}
		\mathrm{dist}_{2}\left(\x,\y\right)=\lV\x-\y \rV_2\cdot\lV\x+\y \rV_2.
	\end{equation}
	Evidently, we have  $\left(\mathrm{dist}_{1}\left(\x,\y\right)\right)^{2}\le\mathrm{dist}_{2}\left(\x,\y\right)$.
	For two non-negative real sequences $\{a_t\}_t$ and $\{b_t\}_t$, we write $b_t = O\left(a_t\right)$ (or $b_t \lesssim a_t$) if $b_t \leq Ca_t$.
	For $\alpha \geq 1$, the $\psi_\alpha$-norm (which is the Orlicz norm taken with respect to function $\exp\left(x^\alpha\right) -1$) of a random variable $X$ is defined as 
	\begin{eqnarray*}
		\lV X\rV_{\psi_\alpha} := \inf\{ t >0: \mathbb{E} \exp(\lv X\rv^\alpha/t^\alpha) \leq 2\}. 
	\end{eqnarray*}
	In particular, $\alpha=2$ gives the sub-Gaussian norm and $\alpha=1$ gives the sub-exponential norm. The random variable $X$ is called sub-Gaussian if $\lV X\rV_{\psi_2}<\infty$ and sub-exponential if $\lV X\rV_{\psi_1}<\infty$.  The Gaussian error function is denoted as $\text{erf}\left(x\right)=\frac{2}{\sqrt{\pi}}\int_{0}^{x}\exp\left(-t^{2}\right)dt.$ 
	
	To describe the structure of possible signals, we mainly introduce tools for characterizing the complexity of set $\mathcal{K}$. If we assume that $\x$ lies in some set $\mathcal{K} \subset B_2^n$ where $B_2^n$ denotes the Euclidean unit ball in $\mathbb{R}^n$, then a key characteristic of structure  of the signal set   
	is its \textit{Gaussian width} $w\left(\mathcal{K}\right)$, defined as  
	\begin{equation}\label{Gaussian width}
		w\left(\mathcal{K}\right) := \mathbb{E} \sup_{\x \in \mathcal{K}} \langle  \pmb{g} , \x \rangle,
	\end{equation}
	where $\pmb{g}\sim\mathcal{N}\left(0,\pmb{I}_{n}\right)$.  In particular, denote $\text{cone}\left(\mathcal{K}\right):=\{tx:t\geq0, x\in \mathcal{K}\}$, then $w^2\left(\text{cone}\left(\mathcal{K}\right)\cap \mathbb{S}^{n-1}\right)$ is a meaningful approximation for dimension \cite{vershynin2018high}, it can also be said that $w^2\left(\mathcal{K}\right)$ can serve as the \textit{effective dimension} of set $\mathcal{K}$.	
	We point out some common examples of Gaussian width in the literature.    
			 If $\mathcal{K} = B_2^n$ or $\mathcal{K} = \mathbb{S}^{n-1}$ then $w\left(\mathcal{K}\right) =\mathcal{O}\left(\sqrt{n}\right)$.
		If $\mathcal{K}=\mathcal{K}_{n,s}:=\{\x\in\mathbb{R}^{n}:\lV\x\rV_{2}=1,\lV\x\rV_{1}=\sqrt{s}\}$, then $w\left(\mathcal{K}\right)=\mathcal{O}\left(\sqrt{s\log\left(en/s\right)}\right)$. The relationship between Gaussian width and covering number can be obtained by Sudakov's minoration inequality in $\mathbb{R}^{n}$  \cite{talagrand2005generic}. Let $\mathcal{N}\left(\mathcal{K},\delta\right)$ be the $\delta$-covering number of set $\mathcal{K}$ then for any $\delta\textgreater0$, we have
	\begin{eqnarray}\label{Sudakov}
		w\left(\mathcal{K}\right)\ge c\delta\sqrt{\log \mathcal{N}\left(\mathcal{K},\delta\right)}.
	\end{eqnarray}
It can be seen that Gaussian width can provide an upper bound for covering number.

\subsection{Organization }
The rest of this paper is organized as follows. 
The main results are presented in Section \ref{results}. In this section, we provide recovery guarantees based on the amplitude-based nonlinear LAD model (\ref{model1}) and intensity-based nonlinear LAD model (\ref{model3}). 
In Section \ref{frame}, we give a framework for adversarial robust recovery in phase retrieval.
In Section \ref{nonROBC}, we introduce the notion of nonlinear ROB condition and show that it holds with high probability 
for Gaussian measurements.
We provide the proofs for our main results in Section \ref{main results}. 
Proof of 
Dvoretzky-Kiefer-Wolfowitz type inequality is presented in Appendix  \ref{conin} and related probability density function is presented in Appendix \ref{RV}. Partial proofs of nonlinear ROB condition are presented in Appendix \ref{partial}.
	
\section{Main Results}\label{results}
	We mainly focus on the amplitude-based nonlinear LAD model (\ref{model1}) and  intensity-based nonlinear LAD model (\ref{model3}), and provide corresponding sharp thresholds for the adversarial corruption fraction. 
	Our theoretical conclusion aligns with the empirical evidence, indicating that the amplitude-based nonlinear LAD model has superior adversarial robustness in comparison to the intensity-based one. To this end, we assume a tractable model in which the design vectors $\{ \pmb{a}_i \}_i$ are i.i.d standard Gaussian vectors. 	
	\subsection{Adversarial Robust Recovery for Amplitude Measurement}
	In this subsection, we provide theoretical results for the nonlinear amplitude-based LAD model (\ref{model1}). Our results show that 
	the solution of (\ref{model1}) is robust to any fraction of adversarial corruption $s\textless s^{*,1}\approx 0.2043$. 
	\begin{theorem} \label{theorem1}
		Suppose that $s\textless s^{*,1}-2\xi$ and $s=\lV\pmb{z}\rV_{0}/m$, where $s^{*,1}\approx 0.2043$. If the
		number of measurements satisfies 
		\begin{eqnarray}\label{measurement}
			m \gtrsim \xi^{-4}\cdot w^2\left(\text{cone}\left(\mathcal{K}\right)\cap \mathbb{S}^{n-1}\right),
		\end{eqnarray} 
		then for all $\pmb{x}_0\in \mathbb{R}^{n}$ with probability at least $1-\mathcal{O}\left(e^{-cm\xi^{2}}\right)$, the solution $\x_{\star}$ of (\ref{model1}) satisfies
		\begin{eqnarray}\label{noise}
	\mathrm{dist}_{1}\left(\x_{\star},\x_{0}\right)\leq C\frac{\lV\pmb{\omega}\rV_1}{m},	
			\end{eqnarray}
		where $C,c$ are positive numerical constants related to adversarial corruption fraction $s$. 
	\end{theorem}
	\begin{remark}\label{re3}
	By (\ref{noise}) we have $	\mathrm{dist}_{1}\left(\x_{\star},\x_{0}\right)\lesssim\frac{\lV\pmb{\omega}\rV_{1}}{m}\le\frac{\lV\pmb{\omega}\rV_{2}}{\sqrt{m}}$.
	Thus, the estimation error in our nonlinear LAD model has a better error bound than LS model. 
		\end{remark}	
	\begin{remark}\label{re2}
	In the case of $\pmb{\omega}=0$, the ground-truth $\x_{0}$ can be exactly reconstructed via (\ref{model1}) 
		even adversarial corruption fraction approaches $s^{*,1}\approx 0.2043$. 
	If we choose $\mathcal{K}=\mathbb{R}^{n}$ then we have $m=\mathcal{O}\left(n\right)$ and if $\mathcal{K}=\mathcal{K}_{n,s}$ then $m=\mathcal{O}\left(s\log\left(en/s\right)\right)$. Thus the number of measurement (\ref{measurement}) is nearly optimal in many cases. 
	\end{remark}
	In the following theorem, we also illustrate that in the case of dense noise $\pmb{\omega}=0$, (\ref{model1}) exhibits a significant inability to accurately reconstruct the ground-truth $\x_{0}$ when adversarial sparse fraction $s>s^{*,1}$.
	It indicates that the threshold $s^{*,1}\approx 0.2043$ is sharp.

\begin{theorem} \label{theorem3}
		Set $s\textgreater s^{*,1}+2\xi$ where $s^{*,1}\approx 0.2043$. If $\pmb{\omega}=\pmb{0}$ and $\mathcal{K}=\mathbb{R}^{n}$, then for any $\x_0\in \mathbb{R}^{n}$, there exists  an adversarial sparse outlier $\pmb{z}$ with $\Vert \pmb{z}\Vert_{0}=sm$, such that the solution of (\ref{model1}) is not exactly the ground-truth $\x_0$ with probability exceeding $1-\mathcal{O}\left(e^{-cm\xi^{2}}\right)$.	
	\end{theorem}	
	
	\subsection{Adversarial Robust Recovery for Intensity Measurement}
	We provide theoretical results for the nonlinear intensity-based LAD model (\ref{model3}) where the solution of (\ref{model3}) is robust to any fraction of corruptions $s\textless s^{*,2}\approx 0.1185$ and in the case $\pmb{\omega}=0$, (\ref{model3}) fails to exactly recover the vector $\x_{0}$ when $s$ exceeding $s^{*,2}$.
	
	\begin{theorem} \label{theorem2}
		Suppose that $s\textless s^{*,2}-2\xi$ and $s=\lV\pmb{z}\rV_{0}/m$, where $s^{*,2}\approx 0.1185$. If the
		number of measurements satisfies 
		\begin{eqnarray*}
			m \gtrsim \xi ^{-4}\cdot w^2\left(\text{cone}\left(\mathcal{K}\right)\cap \mathbb{S}^{n-1}\right),
		\end{eqnarray*} 
		then for all $\pmb{x}_0\in \mathbb{R}^{n}$ with probability at least $1-\mathcal{O}\left(e^{-cm\xi^{2}}\right)$, the solution $\x_{\star}$ of (\ref{model3}) satisfies
		\begin{eqnarray}\label{d2}
                           \mathrm{dist}_{2}\left(\x_{\star},\x_{0}\right)\leq C\frac{\lV\pmb{\omega}\rV_1}{m},	
			\end{eqnarray}
		where $C$ and $c$ are positive numerical constants related to $s$. 	
	\end{theorem}
		\begin{remark}
By (\ref{d2}), we have $\mathrm{dist}_{1}\left(\x_{\star},\x_{0}\right)\lesssim\frac{\sqrt{\lV\pmb{\omega}\rV_1}}{m^{1/2}}$. Besides, we can confirm that 
$  \mathrm{dist}_{1}\left(\x_{\star},\x_{0}\right)\cdot\lV\x_{0}\rV_{2}\le\mathrm{dist}_{2}\left(\x_{\star},\x_{0}\right)$. Thus, we can get 
\begin{eqnarray}\label{bound2}
 \mathrm{dist}_{1}\left(\x_{\star},\x_{0}\right)\lesssim\min\left\{\frac{\sqrt{\lV\pmb{\omega}\rV_1}}{m^{1/2}},\frac{\lV\pmb{\omega}\rV_1}{\lV\x_{0}\rV_{2}\cdot m}\right\}.
\end{eqnarray}
Evidently, error bound (\ref{bound2}) is tighter than LS model in \cite{huang2024performance}.

		\end{remark}
			\begin{remark}
	\cite{huang2022outlier} considered the Robust-PhaseLift model and also obtained the threshold $s^{*,2}\approx 0.1185$ when the signs of adversarial sparse outlier $\z$ are drawn from the Rademacher distribution and the support of $\z$ is fixed. Our theoretical results of model (\ref{model3}) are not dependent on these assumptions.
		\end{remark}
	
	\begin{theorem} \label{theorem4}
		Set $s\textgreater s^{*,2}+2\xi$ where $s^{*,2}\approx 0.1185$. If $\pmb{\omega}=\pmb{0}$ and $\mathcal{K}=\mathbb{R}^{n}$, for any $\x_0\in \mathbb{R}^{n}$, there exists  an adversarial
		sparse outlier $\pmb{z}$ with $\Vert \pmb{z}\Vert_{0}=sm$, such that the solution of (\ref{model3}) is not exactly the ground-truth $\x_0$ with probability exceeding $1-\mathcal{O}\left(e^{-cm\xi^{2}}\right)$.	
	\end{theorem}

\section{A Framework for Adversarial Phase Retrieval}\label{frame}	
	\subsection{ Nonlinear Robust Outlier Bound Condition}\label{robc}
For simplification, the corrupted measurements are written as
	\begin{equation*}
		\pmb{b}=\lv\A\x_{0}\rv^{k}+\pmb{\omega}+\pmb{z},\quad \|\pmb{z}\|_{0}\leq sm,\quad k=1,2.
	\end{equation*}
		In this section
		 we have considered a unified form for amplitude measurement and intensity measurement without distinction.
	Due to the adversarial sparse outliers $\pmb{z}$'s dependence on the measurement system $(\pmb{b}, \pmb{A})$ and ground-truth $\x_0$, 
	as well as the the nonlinearity of operators $\lv\mA\x\rv$ and $\lv\mA\x\rv^{2}$, 
	we give the definition of nonlinear ROB condition.
		
	\begin{definition}[Nonlinear Robust Outlier Bound Condition]
		For amplitude measurement ($k=1$) or intensity measurement ($k=2$), and adversarial corruption friction $s$, sampling matrix $\mA\in \mathbb{R}^{m\times n}$ is said to satisfy the nonlinear robust outlier bound condition, 
		if there exists a constant $\mathcal{C}\left(s\right)> 0$ depending on $s\in \left[0,1\right]$,
		such that the following holds for all vectors $\pmb{x},\pmb{y}\in \mathcal{K}$,
		\begin{eqnarray*}
			\frac{1}{m}\min_{\left\{\mS \subset\left[m\right],  |\mS|\leq s m\right\} }
			\left[\lV \lv\mA_{{\mS}^{c}} \x\rv^{k} -\lv\mA_{{\mS}^{c}} \y\rv^{k} \rV_1-\lV \lv\mA_{\mS} \x\rv^{k} -\lv\mA_{\mS} \y\rv^{k} \rV_1 \right]
			\ge \mathcal{C}\left(s\right)\cdot\mathrm{dist}_{k}\left(\x,\y\right).
		\end{eqnarray*}
	\end{definition}
	Nonlinear ROB condition characterizes the worst-case scenario where $\lv\mA\x\rv^{k}-\lv\mA\y\rv^{k}$ are corrupted in the sense of the $\ell_{1}$-norm under the given adversarial corruption fraction $s$.
	
	\subsection{Adversarial Sparse Outlier Separation Condition}\label{ASOSC}
	If the sampling matrix $\mA$ satisfies the nonlinear ROB condition, we then can give an upper bound to the distance between the model solution $\x_{\star}$ and the ground-truth 
	$\x_{0}$ by the subsequent adversarial sparse outlier separation condition.
	\begin{lemma}[Adversarial Sparse Outlier Separation Condition]\label{ASOSC1}
		For adversarial corruption fraction $s$, assume the sampling matrix $\mA$ satisfies the nonlinear ROB condition with coefficient $\mathcal{C}\left(s\right) >0$, 
		then the solution $\x_{\star}$ of nonlinear LAD models (\ref{model1}) or (\ref{model3}) satisfies
		\begin{equation}
		\mathrm{dist}_{k}\left(\x_{\star},\x_{0}\right)\le \widetilde{\mathcal{C}}\left(s \right)\frac{\lV\pmb{\omega}\rV_1}{m},
		\end{equation}	
		where $\widetilde{\mathcal{C}}\left(s\right)=2/\mathcal{C}\left(s\right)$.
	\end{lemma}
	\begin{proof}
		By definition of $\x_{\star}$ and $ \pmb{b}=\lv\mA\x_{0}\rv^{k}+\pmb{\omega}+\pmb{z}$, we have
		\begin{eqnarray*}
			\lV \lv\mA \x_{\star}\rv^{k} -\lv\mA \x_{0}\rv^{k} - \pmb{z} -\pmb{\omega } \rV_1 \leq \lV \pmb{z}+\pmb{\omega }  \rV_1.
		\end{eqnarray*}
		Since $\ell_{1}$-norm can be divided into  parts $\mathcal{S}$ and $\mathcal{S}^{c}$, we can get 
		\begin{eqnarray*}
			\lV \lv\mA_{\mS} \x_{\star}\rv^{k} -\lv\mA_{\mS} \x_{0}\rv^{k} - \pmb{z} -\pmb{\omega}_{\mS} \rV_1 
			+
			\lV \lv\mA_{{\mS}^{c}} \x_{\star}\rv^{k} -\lv\mA_{{\mS}^{c}} \x_{0}\rv^{k} -\pmb{\omega}_{{\mS}^{c}} \rV_1\
			\leq \lV \pmb{z}+\pmb{\omega}_{\mS}  \rV_1+\lV\pmb{\omega}_{{\mS}^{c}}  \rV_1.
		\end{eqnarray*}
		The triangle inequality yields that
		\begin{eqnarray*}
			\lV \lv\mA_{\mS} \x_{\star}\rv^{k} -\lv\mA_{\mS} \x_{0}\rv^{k} - \pmb{z} -\pmb{\omega}_{\mS} \rV_1
			&\geq& \lV \pmb{z}+\pmb{\omega}_{\mS} \rV_1-\lV \lv\mA_{\mS} \x_{\star}\rv^{k} -\lv\mA_{\mS} \x_{0}\rv^{k} \rV_1,\\
			\lV \lv\mA_{{\mS}^{c}} \x_{\star}\rv^{k} -\lv\mA_{{\mS}^{c}} \x_{0}\rv^{k} -\pmb{\omega}_{{\mS}^{c}} \rV_1 
			&\geq& 	\lV \lv\mA_{{\mS}^{c}} \x_{\star}\rv^{k} -\lv\mA_{{\mS}^{c}} \x_{0}\rv ^{k}\rV_1 - \lV\pmb{\omega}_{{\mS}^{c}}\rV_1.
		\end{eqnarray*}
		Thus, we can further bound
		\begin{equation*}
			\lV \lv\mA_{{\mS}^{c}} \x_{\star}\rv^{k} -\lv\mA_{{\mS}^{c}} \x_{0}\rv^{k} \rV_1-\lV \lv\mA_{\mS} \x_{\star}\rv^{k} -\lv\mA_{\mS} \x_{0}\rv^{k} \rV_1 
			\leq  2\lV\pmb{\omega}_{{\mS}^{c}} \rV_1\leq 2\lV\pmb{\omega}\rV_1.
		\end{equation*}
Since the set $\mS$ should encompass all the subsets of $\left[m\right]$ that satisfy $|\mS|\leq s m$, based on the nonlinear ROB condition, we have 
		\begin{eqnarray*}
			\mathrm{dist}_{k}\left(\x_{\star},\x_{0}\right)\leq \frac{2}{\mathcal{C}\left(s\right)}\frac{\lV\pmb{\omega}\rV_1}{m}=\widetilde{\mathcal{C}}\left(s\right)\frac{\lV\pmb{\omega}\rV_1}{m}.
		\end{eqnarray*}
	\end{proof}

\section{Analysis of Nonlinear ROB Condition}\label{nonROBC}
In this section, we will respectively prove that the sampling matrix $\mA$ satisfies nonlinear ROB condition for the two cases. We further give the sharp sparsity thresholds $s^*$ of adversarial outliers.
	\subsection{Nonlinear ROB Condition for Amplitude Measurement}
	When $\pmb{a}\sim\mathcal{N}\left(0,\pmb{I}_{n}\right)$, we can see that $X:=\langle\pmb{a},\x\rangle\sim \mathcal{N}\left(0,\lV\x\rV^{2}_{2}\right), Y:=\langle\pmb{a},\y\rangle\sim \mathcal{N}\left(0,\lV\y\rV^{2}_{2}\right)$ for fixed $\x, \y \in \mathbb{R}^{n}$.
	Without losing generality, we assume 
	$\lV\x\rV_{2}=1$ and $ \lV\y\rV_{2}=\alpha\in\left[0,1\right]$, and denote the correlation coefficient between $X$ and $Y$ as $\rho$.
	Then we have
		\begin{eqnarray*}
		\lV\lv\mA\x\rv-\lv\mA\y\rv\rV_{1}=\sum_{i=1}^{m}\big|\lv\langle\pmb{a}_{i},\x\rangle\rv-\lv\langle\pmb{a}_{i},\y\rangle\rv\big|=
		\sum_{i=1}^{m}\big| \lv X_{i} \rv - \lv Y_{i} \rv \big|.
	\end{eqnarray*}
	We provide the following lemma to give the PDF of $\lv Z_{\rho,\alpha}\rv:=\big| \lv X \rv - \lv Y \rv \big|$, which serves as the primary statistical tool for our analysis. Since the subsequent analysis solely focuses on the PDF within the context of integration, we only consider the case $z\textgreater 0$.
	\begin{lemma}\label{x-y}
We set two variables $X\sim \mathcal{N}\left(0,1\right),Y\sim \mathcal{N}\left(0,\alpha^{2}\right)$, $ 0\le \alpha \le 1$ and their correlation coefficient $\rho\in \left[-1,1 \right]$.
		Then for $z\textgreater 0$, the PDF of $\lv Z_{\rho,\alpha}\rv  = \bigl| \lv X\rv-\lv Y\rv \bigr|$ is:\\
		If $\lv\rho\rv\textless1,\alpha\neq0$, set $\varLambda_{\pm}=\alpha^{2}\pm2\rho\alpha+1$, then
		{\small
		\bea
		&&g_{\rho,\alpha}\left(z\right)= \\
		&&\ \  \frac{\exp({-\frac{z^{2}}{2\varLambda_{+}}})}{\sqrt{2\pi\varLambda_{+}}} 
			\left[2-\text{erf}\left({\frac{1+\rho\alpha}{\alpha\sqrt{2\varLambda_{+}\left(1-\rho^{2}\right)}} z	}\right)
		-\text{erf}\left(-{\frac{1+\rho\alpha}{\alpha\sqrt{2\varLambda_{+}\left(1-\rho^{2}\right)}} z}+\sqrt{\frac{\varLambda_{+}}{2\left(1-\rho^{2}\right)}}\frac{z}{\alpha}\right)\right]\\
		&&\ \ +\frac{\exp({-\frac{z^{2}}{2\varLambda_{-}}})}{\sqrt{2\pi\varLambda_{-}} } 
			\left[2-\text{erf}\left({\frac{1-\rho\alpha}{\alpha\sqrt{2\varLambda_{-}\left(1-\rho^{2}\right)}}z}\right)
		-\text{erf}\left({-\frac{1-\rho\alpha}{\alpha\sqrt{2\varLambda_{-}\left(1-\rho^{2}\right)}}z}+\sqrt{\frac{\varLambda_{-}}{2\left(1-\rho^{2}\right)}}\frac{z}{\alpha}\right)	\right]. 
	\eea}
	\\
		If $\rho=\pm1$, then $ g_{\rho,\alpha}\left(z\right)=
		\frac{1}{1-\alpha}\sqrt{\frac{2}{\pi}}\exp\big[{-\frac{z^{2}}{2\left(1-\alpha\right)^{2}}}\big]$.\\
If $\alpha=0$, then $g_{\rho,\alpha}\left(z\right)=\sqrt{\frac{2}{\pi}}\exp\big({-z^{2}/2}\big).$	

	\end{lemma}
	\begin{proof}
		We first consider when $\lv\rho\rv\textless1$ and $\alpha\neq0$.
		By (3.1) in \cite{psarakis2001some}, the PDF of the bivariate folded normal distribution $(\lv X \rv , \lv Y\rv)$ is given by:
		\begin{eqnarray*}
			f_{\lv X\rv,\lv Y\rv}\left(x,y\right)
			=\frac{\exp \left[-\frac{1}{2\left(1-\rho^{2}\right)}  \left(x^{2}-\frac{2\rho xy}{\alpha}+\frac{y^{2}}{\alpha^{2}} \right)\right]+\exp \left[-\frac{1}{2\left(1-\rho^{2}\right)}  \left(x^{2}+\frac{2\rho xy}{\alpha}+\frac{y^{2}}{\alpha^{2}} \right)\right]}{\pi\alpha\sqrt{1-\rho^{2}}}
			\quad x,y\ge0.
		\end{eqnarray*}
		By integral transformation, the PDF of $Z_{\rho,\alpha}=\lv X\rv-\lv Y\rv$ obeys 
		\begin{eqnarray}\label{pdf(x-y)}
			f_{\rho,\alpha}\left(z\right)=\mathbbm{1}_{z\in\left(-\infty,0\right)}\cdot\int_{0}^{+\infty}f_{\lv X\rv,\lv Y\rv}\left(x,x-z\right)dx+ \mathbbm{1}_{z\in\left[0,+\infty\right)}\cdot\int_{z}^{+\infty}f_{\lv X\rv,\lv Y\rv}\left(x,x-z\right)dx.
		\end{eqnarray}
		It can be derived from \cite{gradshteyn2014table} that
		\begin{eqnarray}\label{exp1}
			\int_{0}^{+\infty}\exp\left(-\frac{x^{2}}{4\beta}-\gamma x\right)dx=\sqrt{\pi\beta}\exp(\beta \gamma^{2}) \left[1-\text{erf}(\gamma\sqrt{\beta})\right] \quad \beta\textgreater 0,
		\end{eqnarray}
		and similar calculation yields that
		\begin{eqnarray}\label{exp2}
			\int_{z}^{+\infty}\exp\left(-\frac{x^{2}}{4\beta}-\gamma x\right)dx=\sqrt{\pi\beta}\exp\left(\beta \gamma^{2}\right) \left[1-\text{erf}\left(\gamma\sqrt{\beta}+\frac{z}{2\sqrt{\beta}}\right)\right] \quad \beta\textgreater 0.
		\end{eqnarray}
		We will use (\ref{exp1}) and (\ref{exp2}) to calculate the PDF of $Z_{\rho,\alpha}$.		
		If $z\textless0$, thus 
		\bea
		f_{\rho,\alpha}\left(z\right)=\int_{0}^{+\infty}f_{\lv X\rv,\lv Y\rv}\left(x,x-z\right)dx.
		\eea 
		For convenience, we define
		\begin{eqnarray*}
			-\frac{1}{2\left(1-\rho^{2}\right)}\left[x^{2}\pm\frac{2\rho x\left(x-z\right)}{\alpha}+\frac{\left(x-z\right)^{2}}{\alpha^{2}}\right] 
			:=-\frac{x^{2}}{4\beta_{1(2)}}-\gamma_{1(2)}x-\delta.
		\end{eqnarray*}
		Then we can get
		\begin{eqnarray*}	
		\beta _{1(2)} = \frac{(1-\rho^{2})\alpha^{2}}{2\varLambda_{\pm}} > 0,\quad \gamma_{1(2)}=-\frac{1}{1-\rho^{2}}(\frac{z}{\alpha^{2}}\pm\frac{\rho z}{\alpha})\quad \text{and}\quad \delta=\frac{z^{2}}{2(1-\rho^{2})\alpha^{2}}.
		\end{eqnarray*}
		Moreover, it can be obtained through simple calculation that
		\begin{eqnarray*}
			\beta_{1}\gamma^{2}_{1}
			=\frac{\left(1+\rho\alpha\right)^{2}z^{2}}{2\left(1-\rho^{2}\right)\varLambda_{+}\alpha^{2}}\quad \text{and} \quad
			\beta_{2}\gamma^{2}_{2}=\frac{\left(1-\rho\alpha\right)^{2}z^{2}}{2\left(1-\rho^{2}\right)\varLambda_{-}\alpha^{2}}.
		\end{eqnarray*}
		We now can evaluate the integral:
		\begin{eqnarray*}
			&&f_{\rho,\alpha}\left(z\right)\\
			&=& \frac{\exp\left(-\delta \right)}{\pi\alpha\sqrt{1-\rho^{2}}}\int_{0}^{+\infty}\exp\left(-\frac{x^{2}}{4\beta_{1}}-\gamma_{1}x\right)+\exp\left(-\frac{x^{2}}{4\beta_{2}}-\gamma_{2}x\right)dx\\
			&=& \frac{\sqrt{\beta_{1}}\exp\left(\beta_{1} \gamma_{1}^{2}-\delta\right)}{\alpha\sqrt{\pi\left(1-\rho^{2}\right)}} \left[1-\text{erf}\left(\gamma_{1}\sqrt{\beta_{1}}\right)\right]+ \frac{\sqrt{\beta_{2}}\exp\left(\beta_{2} \gamma_{2} 
			^{2}-\delta\right)}{\alpha\sqrt{\pi\left(1-\rho^{2}\right)}} \left[1-\text{erf}\left(\gamma_{2}\sqrt{\beta_{2}}\right)\right]\\
			&=&\frac{\exp\left(-\frac{z^{2}}{2\varLambda_{+}}\right)}{\sqrt{2\pi \varLambda_{+}}} 
			\left[1-\text{erf}\left(-{\frac{1+\rho\alpha}{\alpha\sqrt{2\left(1-\rho^{2}\right)\varLambda_{+}} }z
			}\right)\right] \\
			&&+\frac{\exp\left(-\frac{z^{2}}{2\varLambda_{-}}\right)}{\sqrt{2\pi\varLambda_{-}}} 
			\left[1-\text{erf}\left(-{\frac{1-\rho\alpha}{\alpha\sqrt{2\left(1-\rho^{2}\right)\varLambda_{-}} }z}\right)\right].
		\end{eqnarray*}	
		If $z\ge0$, then $f_{\rho,\alpha}\left(z\right)=\int_{z}^{+\infty}f_{\lv X\rv,\lv Y\rv}\left(x,x-z\right)dx$. Similarly, we have
		\begin{eqnarray*}
			f_{\rho,\alpha}\left(z\right)
			&=&\frac{\exp\left(-\frac{z^{2}}{2\varLambda_{+}}\right)}{\sqrt{2\pi\varLambda_{+}}} 
			\left[1-\text{erf}\left(-{\frac{1+\rho\alpha}{\alpha\sqrt{2\left(1-\rho^{2}\right)\varLambda_{+} } }z}+\sqrt{\frac{\varLambda_{+}}{2\left(1-\rho^{2}\right)}}\frac{z}{\alpha}\right)\right]\\
			&&+\frac{\exp\left(-\frac{z^{2}}{2\varLambda_{-}}\right)}{\sqrt{2\pi\varLambda_{-}}} 
			\left[1-\text{erf}\left(-{\frac{1-\rho\alpha}{\alpha\sqrt{2\left(1-\rho^{2}\right)\varLambda_{-}} }z}+\sqrt{\frac{\varLambda_{-}}{2\left(1-\rho^{2}\right)}}\frac{z}{\alpha}\right)\right].
		\end{eqnarray*}		
		Due to $g_{\rho,\alpha}\left(z\right) = f_{\rho,\alpha}\left(z\right) +f_{\rho,\alpha}\left(-z\right)$, we have obtained the expression.
		
		When $\rho=\pm1$, then $ Z_{\rho,\alpha}=\lv X\rv-\lv Y\rv=\left(1-\alpha\right)\lv X\rv$. Thus
	$f_{\rho,\alpha}\left(z\right)=
		\frac{1}{1-\alpha}\sqrt{\frac{1}{2\pi}}\exp\left[-\frac{z^{2}}{\left(1-\alpha\right)^{2}}\right].$
		Then $g_{\rho,\alpha}\left(z\right)=\frac{1}{1-\alpha}\sqrt{\frac{2}{\pi}}\exp\left[-\frac{z^{2}}{\left(1-\alpha\right)^{2}}\right]$.
		
		When $\alpha=0$, then $Z_{\rho,\alpha}=\lv X\rv-\lv Y\rv=\lv X\rv$. Thus
		$f_{\rho,\alpha}\left(z\right)=\sqrt{\frac{1}{2\pi}}\exp\left(-z^{2}/2\right)$
		and then $g_{\rho,\alpha}\left(z\right)=\sqrt{\frac{2}{\pi}}\exp\left(-z^{2}/2\right)$.
	\end{proof}

	Furthermore, we set the Cumulative Distribution Function (CDF) of $\lv Z_{\rho,\alpha}\rv$ as
	\begin{eqnarray*}
		G_{\rho,\alpha}\left(t\right) = \int_{0}^{t} g_{\rho,\alpha}\left(z\right) dz.
	\end{eqnarray*}
	Due to symmetry, we can only consider $\rho\ge 0$. Thus, $\mathrm{dist}_{1}\left(\x,\y\right)=\sqrt{1+\alpha^{2}-2\alpha\rho}.$ To characterize the deviation between $\ell_1$-norm of largest $s$-fraction and smallest $(1-s)$-fraction of the random variable $\lv Z_{\rho,\alpha}\rv$, 
	we define the balance function
	\begin{equation*}
		M\left(\rho,\alpha, s\right) = M_{1}\left(\rho,\alpha, s\right)-M_{2}\left(\rho,\alpha, s\right)\\
		:=\frac{\left[ \int_0^{G_{\rho,\alpha}^{-1}\left(1-s\right)} - \int_{G_{\rho,\alpha}^{-1}\left(1-s\right)}^{+\infty} \right] z g_{\rho,\alpha}\left(z\right)dz}{\sqrt{1+\alpha^{2}-2\alpha\rho}}.
	\end{equation*}
For fixed $\x,\y$ with corresponding parameters $\rho,\alpha$ and largest $\ell_{1}$-norm corruption set $\mathcal{S}$ with size $\lv\mathcal{S}\rv=sm$, in the limit of large samples (i.e. $m\to\infty$), we see that a form of strong law of large numbers yields that 
	\begin{eqnarray}
M\left(\rho,\alpha, s\right)=\frac{\mathbb{E}\left[\lV \lv\mA_{{\mS}^{c}} \x\rv -\lv\mA_{{\mS}^{c}} \y\rv \rV_1-\lV \lv\mA_{\mS} \x\rv -\lv\mA_{\mS} \y\rv \rV_1 \right]}{m\cdot\mathrm{dist}_{1}\left(\x,\y\right)}.
\end{eqnarray}
	Nonlinear ROB condition is uniformly established for any $\x,\y\in\mathcal{K}\subset\mathbb{R}^{n}$, 
	we should  take into account the minimum value of function $M\left(\rho,\alpha, s\right)$ regarding parameters $\rho,\alpha\in\left[0,1\right]$.
	It motivates us to analyze the minimum balance function $\mathcal{M}\left(s\right)$.
	\begin{definition}[Minimum Balance Function]	
		\begin{eqnarray}
			\mathcal{M}\left(s\right)=\min_{\rho,\alpha \in[0,1]}\frac{\left[\int_{0 }^{G_{\rho,\alpha}^{-1}\left(1-s\right)}- \int_{G_{\rho,\alpha}^{-1}\left(1-s\right) }^{+\infty} \right]z g_{\rho,\alpha }\left(z\right)dz}{\sqrt{1+\alpha^{2}-2\alpha\rho}}.
		\end{eqnarray}
	\end{definition}
We can use $\mathcal{M}\left(s\right)$
	to establish the uniform lower bound of the nonlinear ROB condition.
If we set $(\rho,\alpha) = \argmin \mathcal{M}(s_{\rho,\alpha})$, then it goes that	
	 \bea
	 \int_{0 }^{G_{\rho,\alpha}^{-1}\left(1-s_{\rho,\alpha}\right)}z g_{\rho,\alpha }\left(z\right)dz = \int_{G_{\rho,\alpha}^{-1}\left(1-s_{\rho,\alpha}\right) }^{+\infty} z g_{\rho,\alpha }\left(z\right)dz.
	 \eea
	  We can get $s^{*,1} = \min\limits_{\rho,\alpha\in [0,1]} s_{\rho,\alpha} \approx 0.2043$ with the help of computer assistance computation, which implies $\mathcal{M}\left( s^{*,1} \right) =0$.
	  Theorem \ref{k=1} states that the lower bound of the nonlinear ROB condition is dominated by $\mathcal{M}\left(s\right)$ 
	for a given adversarial corruption fraction $s$.

\begin{figure}[htbp]
\centering 

\begin{minipage}[b]{0.45\textwidth}
\centering 
\includegraphics[width=1\textwidth]{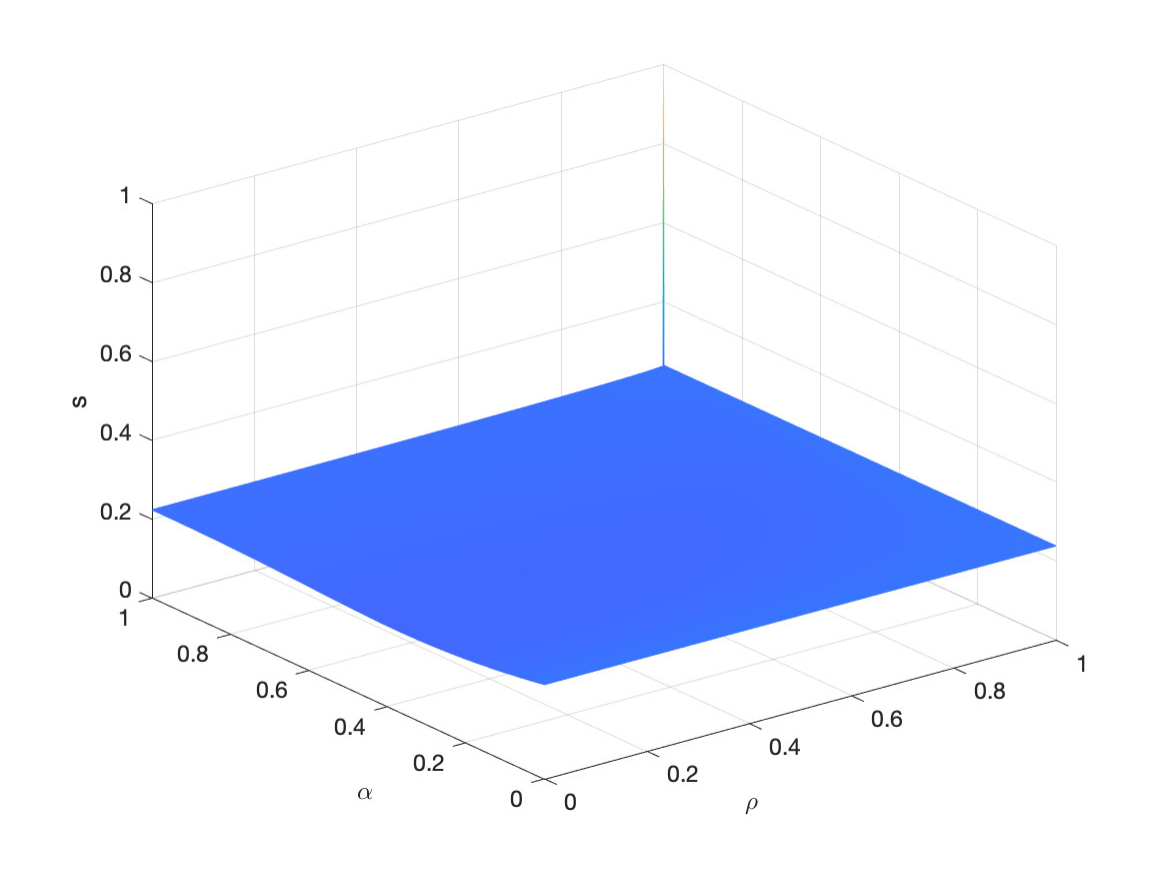} 
\subcaption*{(a) $\alpha, \rho \in [0,1]$.} 
\end{minipage}
\quad
\begin{minipage}[b]{0.45\textwidth} 
\centering 
\includegraphics[width=1\textwidth]{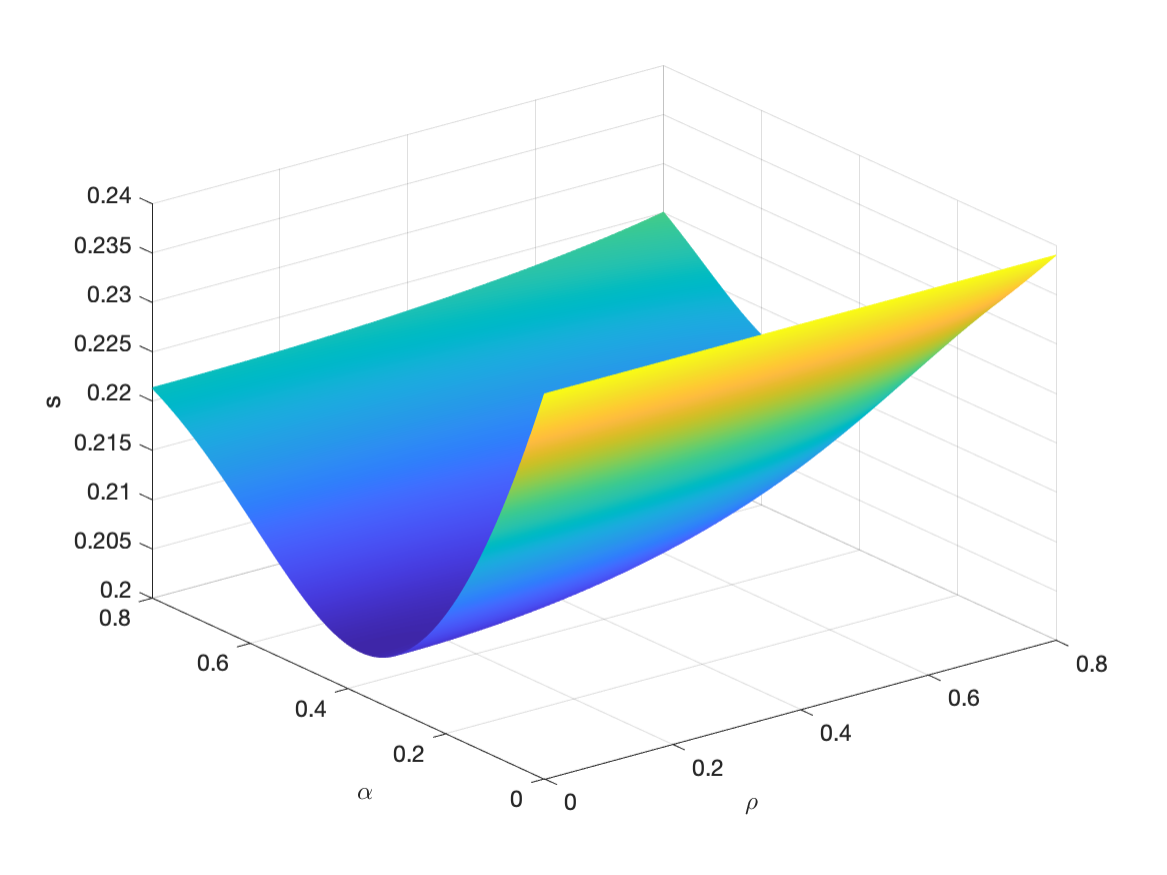}
\subcaption*{(b) $\alpha\in[0,0.8], \rho \in [0.2,0.8]$.}

\end{minipage}
\caption{The threshold $s$ for any $\alpha, \rho \in [0,1]$.}
\label{Fig_sap}
\end{figure}
	
	\begin{theorem}\label{k=1}
		Let 
		$\mA\x=\{\langle\pmb{a}_{1},\x\rangle,\cdots,\langle\pmb{a}_{m},\x\rangle\}$ where $\pmb{a}_{i}\stackrel{i.i.d}{\sim}\mathcal{N}\left(0,\pmb{I}_{n}\right)$.
		Assume that $0 \textless s \textless s^{*,1}\approx0.2043$ and $c_0$ is a positive numerical constant corresponding to $l_{0}$. 
		If
		\begin{eqnarray*}
			m \gtrsim \xi ^{-4}\cdot w^2\left(\text{cone}(\mathcal{K})\cap \mathbb{S}^{n-1}\right),
		\end{eqnarray*}
		then for any $s<s^{*,1}$,  $\mA$ satisfies the $s$-nonlinear ROB condition for all $\x,\y\in \mathcal{K}$:
		\begin{equation}
		\begin{split}
			\frac{1}{m}\min_{\left\{\mS \subset\left[m\right],  |\mS|\leq s m\right\} }
			\Bigl[\lV \lv\mA_{{\mS}^{c}} \x\rv -\lv\mA_{{\mS}^{c}} \y\rv \rV_1&-\lV \lv\mA_{\mS} \x\rv -\lv\mA_{\mS} \y\rv \rV_1 \Bigr]\\
			&\ge\left[\mathcal{M}\left(s+\xi \right)-l_{0}\xi\right]\cdot\mathrm{dist}_{1}\left(\x,\y\right)
			\end{split}
		\end{equation}
		with probability at least $1- \mathcal{O}\left(e^{-c_{0}m\xi^{2}}\right)$.
	\end{theorem}
	\begin{remark}
	The condition $s\textless s^{*,1}$ is to ensure $\mathcal{M}\left(s+\xi \right)-l_{0}\xi\textgreater 0$. In other words, without this condition the above theorem still holds.
	\end{remark}
	By setting $s=0$ and $s=1-\xi$ in the nonlinear ROB condition, we can obtain the stability of phase retrieval with upper and lower bounds. It can be considered as a random measurement version of stability in \cite{bandeira2014saving}. Specifically, we have the following corollary.
	\begin{corollary}\label{coro1}
		In the case $\mathcal{K}=\mathbb{R}^{n}$, if 
		$m \gtrsim \xi ^{-4}\cdot n$, then with probability at least $1- \mathcal{O}\left(e^{-c_{0}m\xi^{2}}\right)$, we have
		\begin{eqnarray}\label{co1}
			\left[\sqrt{\frac{2}{\pi}}\left(2-\sqrt{2}\right)-\tilde{l}\xi\right]
			\cdot\mathrm{dist}_{1}\left(\x,\y\right)
			\le\frac{1}{m}\lV \lv\mA \x\rv -\lv\mA \y\rv \rV_1
			\le \left(  \sqrt{\frac{2}{\pi}}+   \tilde{l}\xi\   \right)\cdot\mathrm{dist}_{1}\left(\x,\y\right).
		\end{eqnarray}
		The lower bound  $\sqrt{\frac{2}{\pi}}\left(2-\sqrt{2}\right)$ and upper bound $\sqrt{\frac{2}{\pi}}$ here are $\mathcal{M}\left(0\right)$ and $-\mathcal{M}\left(1\right)$ given by Proposition \ref{K} in the subsequent subsection.
	\end{corollary}	
		
	\subsection{Proof of Theorem \ref{k=1} and Corollary \ref{coro1}}
	We divide the proof of Theorem \ref{k=1} into three steps. 
	The first step is to construct some properties of the minimum balance function $\mathcal{M}\left(s\right)$. 
	It is then stated that $\mA$ obeys the nonlinear ROB condition for fixed $\x,\y$. Finally, we use the standard covering-number based argument and Sudakov's minoration inequality to extend the result for all $\x,\y \in \mathcal{K}$.
		
	{\bf Step 1: Properties of $\mathcal{M}\left(s\right)$.}
	We will prove that $\mathcal{M}\left(s\right)$ is a well-defined function and has a unique zero $s^{*,1}$.
	\begin{proposition}\label{K}
		The minimum balance function $\mathcal{M}\left(s\right)$ satisfies:
		\begin{itemize}
			\item[$\mathrm{(a)}$]$\mathcal{M}\left(s\right)$ is strictly monotonically decreasing on interval $[0,1]$;
			\item[$\mathrm{(b)}$]$\mathcal{M}\left(0\right)=\sqrt{\frac{2}{\pi}}\left(2-\sqrt{2}\right)$ and $\mathcal{M}\left(1\right)=-\sqrt{\frac{2}{\pi}}$;
			\item[$\mathrm{(c)}$]$\mathcal{M}\left(s\right)$ has a unique zero $s^{*,1}$, that is $\mathcal{M}\left(s^{*,1}\right)= 0$. 
		\end{itemize}
	\end{proposition}
	\begin{proof}
		\noindent(a)
		For monotonicity, let $0\le s_{1} \textless s_{2} \le1$ and $\left(\rho_{1},\alpha_{1}\right)=\argmin\mathcal{M}\left(s_{1}\right)$, $\left(\rho_{2},\alpha_{2}\right)=\argmin\mathcal{M}\left(s_{2}\right)$.
		By definition, 
		\begin{eqnarray*} 
			\mathcal{M}\left(s_{1}\right)=M\left(\rho_{1},\alpha_{1},s_{1}\right) \textgreater M\left(\rho_{1},\alpha_{1},s_{2}\right)\geq M\left(\rho_{2},\alpha_{2},s_{2}\right)=\mathcal{M}\left(s_{2}\right).
		\end{eqnarray*}
	\noindent(b)
	It can be checked that $\lv \lv a \rv-\lv b \rv \rv =  \lv a+b \rv + \lv a-b \rv - \lv a \rv-\lv b \rv$, then we have
	\begin{eqnarray*} 
		\lV\lv\mA\x\rv-\lv\mA\y\rv\rV_{1}&=&\lV\mA\left(\x+\y\right)\rV_{1}+\lV\mA\left(\x-\y\right)\rV_{1}-\lV\mA \x\rV_{1}-\lV\mA\y\rV_{1}\\
		&=&\sum_{i=1}^{m}\left(\lv X_{i}+Y_{i}\rv+\lv X_{i}-Y_{i}\rv-\lv X_{i}\rv-\lv Y_{i}\rv\right).
	\end{eqnarray*}
	For $s=0$ or $1$, we need consider $\mathbb{E}\left(\bigl|\lv X\rv-\lv Y\rv\bigr |\right)$. Thus, let
	\begin{eqnarray*} 
		\varphi \left(\rho,\alpha\right)&=& \frac{\int_0^{+\infty} z g_{\rho,\alpha}(z)dz}{\mathrm{dist}_{1}\left(\x,\y \right)} =\frac{\mathbb{E} \left(\bigl|\lv X \rv-\lv Y \rv\bigr|\right)}{ \mathrm{dist}_{1}\left(\x,\y \right) }
		= \frac{\mathbb{E}\left(\lv X+Y \rv+\lv X-Y \rv-\lv X \rv-\lv Y \rv\right)}{ \mathrm{dist}_{1}\left(\x,\y\right) }\\ 
		&=& \sqrt{\frac{2}{\pi}}\frac{\lV\x+\y\rV_{2}+\lV\x-\y\rV_{2}-\lV\x\rV_{2}-\lV\y\rV_{2}}{\mathrm{dist}_{1}\left(\x,\y\right)}\\
		&=&\sqrt{\frac{2}{\pi}} \frac{\left(1+\alpha^{2}+2\alpha\rho\right)^{1/2}+\left(1+\alpha^{2}-2\alpha\rho\right)^{1/2}-1-\alpha}{\left(1+\alpha^{2}-2\alpha\rho\right)^{1/2}}.
	\end{eqnarray*} 
	Then, 
	\begin{eqnarray*} 
		\frac{\partial \varphi \left(\rho,\alpha\right)}{\partial\rho}
		&=&\sqrt{\frac{2}{\pi}}\frac{\alpha\left(\alpha^{2}+1\right)}{\left(1+\alpha^{2}-2\alpha\rho\right)^{3/2}}\left[\frac{2}{\left(1+\alpha^{2}+2\alpha\rho\right)^{1/2}}-\frac{\alpha+1}{\alpha^{2}+1}\right]\\
		&\ge&\sqrt{\frac{2}{\pi}}\frac{\alpha\left(\alpha^{2}+1\right)}{\left(1+\alpha^{2}-2\alpha\rho\right)^{3/2}}\left(\frac{2}{\alpha+1}-\frac{\alpha+1}{\alpha^{2}+1}\right)\ge0.
	\end{eqnarray*}       
	Therefore, for fixed $\alpha\in\left(0,1\right]$, $ \varphi\left(\rho,\alpha\right)$ is increasing with respect to $\rho$. Furthermore,
	\begin{eqnarray*} 
		\frac{\partial \varphi \left(0,\alpha\right)}{\partial\alpha}=\sqrt{\frac{2}{\pi}}\frac{\alpha-1}{\left(1+\alpha^{2}\right)^{3/2}}\le0.
	\end{eqnarray*} 
	Then we can get 
	\begin{eqnarray*} 
		\mathcal{M}\left(0\right)&=&\min_{\rho,\alpha} \varphi \left(\rho,\alpha\right) = \min_{\alpha}\varphi\left(0,\alpha\right) =\varphi \left(0,1\right)=\sqrt{\frac{2}{\pi}}\left(2-\sqrt{2}\right),\\
		\mathcal{M}\left(1\right)&=&-\max_{\rho,\alpha} \varphi \left(\rho,\alpha\right) = -\max_{\alpha}\varphi\left(1,\alpha\right) =
		-\varphi \left(1,\alpha\right)=-\sqrt{\frac{2}{\pi}}.\\
	\end{eqnarray*} 
	\noindent(c) 
	 The continuity of $\mathcal{M}\left(s\right)$ is obvious. By Proposition \ref{K}.(a) and Proposition \ref{K}.(b), $\mathcal{M}\left(s\right)$ is bounded, thus it is well-defined on $\left[0,1\right]$. Then by monotonicity of $\mathcal{M}\left(s\right)$, the uniqueness of zero can be obtained. 
		\end{proof}	
	{\bf Step 2: Proof for fixed vectors.} In this step, we will establish nonlinear ROB condition for fixed $\x,\y$.
	\begin{lemma}\label{A}
		Fix $\x,\y\in \mathcal{K}$, let $\rho=\frac{\langle\x, \y\rangle}{\lV\x\rV_{2}\cdot\lV\y\rV_{2}}$ and $\alpha=\lV\y\rV_{2}/\lV\x\rV_{2}$. 
		For $0 \textless s\textless s^{*,1}$, there exists a positive numerical constant $c_{1}$ corresponding to $l_{1}$ such that with probability exceeding  $1-\mathcal{O}(e^{-c_{1}m\xi^{2}})$,	
		\begin{eqnarray}
			\min_{\mS \subset[m],  |\mS|\leq s m}\frac{\lV\lv\mA_{{\mS}^{c}}\pmb{x}\rv-\lv\mA_{{\mS}^{c}}\pmb{y}\rv\rV_{1}}{\mathrm{dist}_{1}\left(\x,\y\right)}
			&\geq& m\left[M_{1}(\rho,\alpha, s+\xi)-l_{1}\xi\right],\label{A1}\\	
			\max_{\mS \subset[m],  |\mS|\leq s m}\frac{\lV\lv\mA_{\mS}\pmb{x}\rv-\lv\mA_{\mS}\pmb{y}\rv\rV_{1}}{\mathrm{dist}_{1}\left(\x,\y\right)}
			&\leq& m \left[M_{2}(\rho,\alpha, s+\xi)+l_{1}\xi \right].\label{A2}
		\end{eqnarray}
	\end{lemma}
	\begin{proof}	
		Denote $h_{1}\left(x\right) = x \cdot \mathbbm{1}_{[0,t]}(x)$, and $\Gamma \left(t\right) = h_{1}\left(\frac{\lv Z_{\rho,\alpha}\rv}{\sqrt{1+\alpha^{2}-2\alpha\rho}}\right)$. 
		We first prove that for all $\rho,\alpha\in\left[0,1\right]$ and $t \in \mathbb{R}^{+}\cup \{+\infty\}$, 
		there exists $ K_{0}>0$ such that $	\lV\Gamma\left(t\right) \rV_{\psi_{2}} \leq K_{0}$.
		Since $\Gamma \left(t\right) $ is bounded, thus for fixed $t$, $ \| \Gamma \left(t\right)  \|_{\psi_{2}}< \infty $.
		When $t=+\infty $, we claim that $\Gamma \left(t\right) =\frac{\lv Z_{\rho,\alpha}\rv}{\sqrt{1+\alpha^{2}-2\alpha\rho}}$ is sub-Gaussian random variable. Actually, by the definition of $\lV\cdot\rV_{\psi_{2}}$
		\begin{eqnarray*}
			\lV\frac{\lv Z_{\rho,\alpha}\rv}{\sqrt{1+\alpha^{2}-2\alpha\rho}}\rV_{\psi_{2}}=\lV\frac{\lv \langle\pmb{a},\x\rangle\rv-\lv \langle\pmb{a},\y\rangle\rv}{\mathrm{dits}_{1}\left(\x,\y\right)}\rV_{\psi_{2}}
			\le 	\lV\langle\pmb{a},\frac{\x-\y}{\lV\x-\y\rV_{2}}\rangle\rV_{\psi_{2}}=\lV X\rV_{\psi_{2}}< \infty,
		\end{eqnarray*}
		where $X\sim\mathcal{N}\left(0,1\right)$.
		We then consider  the mapping
		\begin{eqnarray*}
			\mathcal{F}:\mathbb{R}\cup\{\infty \}&\longrightarrow&\mathbb{R^{+}}\cup\{+\infty \}\\
			t&\longmapsto&	 \lV \Gamma \left(t\right) \rV_{\psi_{2}}.
		\end{eqnarray*}	
		By the definition of sub-Gaussian norm, we verified that $\mathcal{F}$ is a continuous mapping.
		Apart from this, the mapping
		\begin{eqnarray*}
			\mathcal{G}:S^{1}&\longrightarrow&\mathbb{R}\cup\{\infty \}\\
			e^{i\theta}&\longmapsto&	 \frac{1+\tan\left(\theta/2\right)}{1-\tan\left(\theta/2\right)}:= t
		\end{eqnarray*}		
		establishes the homeomorphism between $S^{1}$ and $\mathbb{R}\cup\{\infty \}$, that is $S^{1}\cong \mathbb{R}\cup\{\infty \}$.
		We then direct our attention towards the composite mapping
		\begin{eqnarray*}
			\mathcal{F}\circ\mathcal{G}:S^{1}&\longrightarrow&\mathbb{R^{+}}\cup\{+\infty \}\\
			e^{i\theta}&\longmapsto&	 \lV \Gamma\left(t\right)\rV_{\psi_{2}}.
		\end{eqnarray*}
		While $\mathcal{G}^{-1}\left(\mathbb{R^{+}}\cup\{+\infty \}\right)=[-\frac{\pi}{2},\frac{\pi}{2}]$,
		then $\mathcal{F}\circ\mathcal{G}$ is a continuous mapping on compact set $[-\frac{\pi}{2},\frac{\pi}{2}]$.
		Thereby $\mathcal{F}\circ\mathcal{G}$ is a bounded mapping, which means that there exists
		$ K_{0} >0$, such that $\| \Gamma\left(t\right)\|_{\psi_{2}} \leq K_{0}$ uniformly.
						
		Let $r\textgreater s\in (0,1)$ be a fixed constant and $t = G_{\rho,\alpha}^{-1}\left(1-r\right)$. 
		We consider the sampling set $\Gamma=\{\Gamma_{1}\left(t\right),\cdots,\Gamma_{m}\left(t\right)\}$ and get 
		\begin{eqnarray*}
			\mathbb{E}\left[\Gamma \left(t\right)\right]=\frac{\int_{0 }^{t}z g_{\rho,\alpha }\left(z\right)dz}{\sqrt{1+\alpha^{2}-2\alpha\rho}}=M_{1}(\rho,\alpha,r).
		\end{eqnarray*}	
		Thus by Hoeffding inequality in \cite{vershynin2018high}, we derive
		\begin{eqnarray}\label{Y1}
			\mathbb{P}\left(\lv {1\over m} \sum_{i=1}^m \Gamma_i\left(t\right) - \mathbb{E}\left[\Gamma \left(t\right)\right] \rv >\varepsilon_1\right) \leq 2e^{- c_{0}\,m\varepsilon_1^2/K_{0}^{2},}.
		\end{eqnarray} 
		
		Based on the Dvoretzky-Kiefer-Wolfowitz type inequality in Lemma \ref{DKW}, 
		by setting $\epsilon = r-s \in (0, 1)$ and $\eta=1-s \in \left(0, 1\right)$, we have
		\begin{eqnarray*}
			t = G_{\rho,\alpha}^{-1}\left(1-r\right) \leq \widehat{G}_{\rho,\alpha}^{-1}\left(1-s\right)
		\end{eqnarray*} 
		with probability exceeding $1-4e^{-2m\left(r-s\right)^{2}}$.
		Since $\widehat{G}_{\rho,\alpha}(z)$ and $\widehat{G}_{\rho,\alpha}^{-1}\left(z\right)$ are monotonically increasing function, we can get
		\begin{eqnarray*}
			\widehat{G}_{\rho,\alpha}(t) \leq 1-s
		\end{eqnarray*} 
		with probability  at least $1 - 4e^{-2m\left(r-s\right)^2}$. 
		Thus, we can know at most $(1-s)$-fraction of the samples lie in $[0, t]$ with probability  at least $1 - 4e^{-2m(r-s)^2}$, 
		and the samples lying in $[0, t]$ are smaller than those of the remaining samples.
		Therefore, we have 
		\begin{eqnarray}\label{Y2}
			\min_{\mS \subset[m],  |\mS|\leq s m}\frac{\lV\lv\mA_{{\mS}^{c}}\pmb{x}\rv-\lv\mA_{{\mS}^{c}}\pmb{y}\rv\rV_{1}}{\mathrm{dist}_{1}\left(\x,\y\right)}
			\ge \sum_{i=1}^m \Gamma_i\left(t\right)
		\end{eqnarray} 
		with probability  at least $1 - 4e^{-2m\left(r-s\right)^2}$ due to the left side of the above inequality represents the smallest $\left(1-s\right)m$ samples in the sampling set $\Gamma$.
		
		Combining (\ref{Y1}) with (\ref{Y2}), we get
		\begin{eqnarray*}
			\frac{1}{m}\min_{\mS \subset[m],  |\mS|\leq s m}\frac{\lV\lv\mA_{{\mS}^{c}}\pmb{x}\rv-\lv\mA_{{\mS}^{c}}\pmb{y}\rv\rV_{1}}{\mathrm{dist}_{1}\left(\x,\y\right)}
			\ge \frac{1}{m}\sum_{i=1}^m  \Gamma_i\left(t\right) \ge \mathbb{E}\left[\Gamma \left(t\right)\right] - \varepsilon_1 = M_{1}\left(\rho,\alpha,s+\epsilon\right)-\varepsilon_1
		\end{eqnarray*}
		with certain probability. 
		By setting $\epsilon =\xi$ and $\varepsilon_1= l_{1}\xi$,
		we finally get (\ref{A1}) with probability  exceeding $1 -\mathcal{O}\left(e^{-c_{1}m\xi^{2}}\right)$.
		
		Similarly, we can establish (\ref{A2}) if we 
		set $h_{2}\left(x\right) = x\cdot \mathbbm{1}_{\left[t,+\infty\right)}\left(x\right)$ and $\Gamma\left(t\right) =h_{2}\left(\frac{\lv Z_{\rho,\alpha}\rv}{\sqrt{1+\alpha^{2}-2\alpha\rho}}\right)$.
	\end{proof}
	
	{\bf Step 3: Uniform  argument.}
	By Lemma \ref{A},  for fixed $\x_{0},\y_{0}\in\mathcal{K}$ with parameters $\rho_{0}$ and $\alpha_{0}$, with probability at least $1 - \mathcal{O}\left(e^{- c_{1} m\xi^2}\right)$ we have that
	\begin{eqnarray*}\label{A3}
	\begin{split}
		&\min_{\mS \subset[m],  |\mS|\leq s m} \lV\lv\mA_{{\mS}^{c}}\x_{0}\rv-\lv\mA_{{\mS}^{c}}\y_{0}\rv\rV_{1}-\lV\lv\mA_{{\mS}}\x_{0}\rv-\lv\mA_{{\mS}}\y_{0}\rv\rV_{1}\\
		&= \min_{\mS \subset[m],  |\mS|\leq s m} \left[\sum_{i\in{\mS}^{c}} \bigl|\lv \langle\pmb{a}_{i}, \pmb{x}_{0}\rangle \rv-\lv\langle\pmb{a}_{i}, \pmb{y}_{0}\rangle \rv\bigr|-\sum_{i\in\mS} \bigl|\lv \langle\pmb{a}_{i}, \pmb{x}_{0}\rangle\rv-\lv \langle\pmb{a}_{i}, \pmb{y}_{0}\rangle \rv\bigr|\right]\\
		&\ge m\left[M_{1}\left(\rho_{0},\alpha_{0},s+\xi\right)-M_{2}\left(\rho_{0},\alpha_{0},s+\xi\right)-2l_{1}\xi\right]\cdot\mathrm{dist}_{1}\left(\x_{0},\y_{0}\right)\\
		&\ge m\left[\mathcal{M}\left(s+\xi\right)-2l_{1}\xi\right]\cdot\mathrm{dist}_{1}\left(\x_{0},\y_{0}\right).
		\end{split}
	\end{eqnarray*}
	
	We divide the proof into three types of situations. 
	
  	\textbf{Case 1}:  In this case we assume that $\lV\x-\y\rV_{2}\le\lV\x+\y\rV_{2}\le5\lV\x-\y\rV_{2}$.  We can let $\lV\x\rV_{2}\ge\lV\y\rV_{2}$, then by standardization, set $\x\in\mathcal{K}\cap\mathcal{S}^{n}$ and $\y\in \mathcal{K}\cap B_2^n$. Then we let $\mathcal{K}^{1}_{\delta },\mathcal{K}^{2}_{\delta }$ be the $\delta$-net of $\text{cone}(\mathcal{K})\cap \mathbb{S}^{n-1}$ and $\text{cone}(\mathcal{K})\cap B_2^n$. 
	Now for all $\x\in\mathcal{K}\cap\mathcal{S}^{n}$ and $\y\in \mathcal{K}\cap B_2^n$, there exist $\x_{0}\in\mathcal{K}^{1}_{\delta }$ and $\y_{0}\in\mathcal{K}^{2}_{\delta }$, such that $\lV\x-\x_{0}\rV_{2}\le\delta, \lV\y-\y_{0}\rV_{2}\le\delta$. 
		
	Set $\mS_{1}$ be the index set of the largest $s$-fraction of $\mS_{\pmb{x},\pmb{y}}=\{\lv\langle\pmb{a}_{i}, \pmb{x}\rangle\rv -\lv\langle\pmb{a}_{i}, \pmb{y}\rangle\rv | i\in[m] \}$ in absolute value.  
	Similarly, $\mS_{2}$ be the subset of $\mS_{\pmb{x}_{0},\pmb{y}_{0}}$,  which collects the index of the largest $s$-fraction of $\mS_{\pmb{x}_{0},\pmb{y}_{0}}$ in absolute value.
	Then, we have
	\begin{eqnarray*}
	\begin{split}
		\min_{\mS \subset[m],  |\mS|\leq s m}
		&\lV\lv\mA_{{\mS}^{c}}\x\rv-\lv\mA_{{\mS}^{c}}\y\rv\rV_{1}\\
		=&\lV\lv\mA_{{\mS}_{1}^{c}}\x\rv-\lv\mA_{{\mS}_{1}^{c}}\y\rv\rV_{1}
		=\sum_{i\in{\mS}_{1}^{c}} \bigl| \lv \langle\pmb{a}_{i}, \pmb{x}\rangle\rv -\lv\langle\pmb{a}_{i}, \pmb{y}\rangle \rv\bigr| \\
		\ge&\sum_{i\in{\mS}_{1}^{c}} \bigl| \lv \langle\pmb{a}_{i}, \pmb{x}_{0}\rangle\rv -\lv\langle\pmb{a}_{i}, \pmb{y}_{0}\rangle \rv\bigr| -
		\sum_{i\in{\mS}_{1}^{c}} \bigl|\lv \langle\pmb{a}_{i}, \pmb{x}\rangle\rv-\lv \langle\pmb{a}_{i}, \pmb{x}_{0}\rangle\rv -\lv\langle\pmb{a}_{i}, \pmb{y}\rangle\rv +\lv \langle\pmb{a}_{i}, \pmb{y}_{0}\rangle\rv \bigr| \\
		\ge&\sum_{i\in{\mS}_{2}^{c}} \bigl|\lv \langle\pmb{a}_{i}, \pmb{x}_{0}\rangle\rv -\lv\langle\pmb{a}_{i}, \pmb{y}_{0}\rangle \rv\bigr|-
		\lV\lv \mA_{{\mS}_{1}^{c}}\x\rv-\lv \mA_{{\mS}_{1}^{c}}\x_{0}\rv -\lv\mA_{{\mS}_{1}^{c}}\y\rv +\lv\mA_{{\mS}_{1}^{c}}\y_{0}\rv\rV_{1}\\
		=&	\min_{\mS \subset[m],  |\mS|\leq s m} \lV\lv\mA_{{\mS}^{c}}\x_{0}\rv-\lv\mA_{{\mS}^{c}}\y_{0}\rv\rV_{1}
		-\lV\lv \mA_{{\mS}_{1}^{c}}\x\rv-\lv \mA_{{\mS}_{1}^{c}}\x_{0}\rv -\lv\mA_{{\mS}_{1}^{c}}\y\rv +\lv\mA_{{\mS}_{1}^{c}}\y_{0}\rv\rV_{1}.\\
		\end{split}
	\end{eqnarray*}
	Similarly, we can get
	\begin{eqnarray*}
	\begin{split}
		&\max_{\mS \subset[m],  |\mS|\leq s m} \lV\lv\mA_{{\mS}}\x\rv-\lv\mA_{{\mS}}\y\rv\rV_{1}\\
		&\le\max_{\mS \subset[m],  |\mS|\leq s m} \lV\lv\mA_{{\mS}}\x_{0}\rv-\lv\mA_{{\mS}}\y_{0}\rv\rV_{1}
		+\lV\lv \mA_{{\mS}_{1}}\x\rv-\lv \mA_{{\mS}_{1}}\x_{0}\rv -\lv\mA_{{\mS}_{1}}\y\rv +\lv\mA_{{\mS}_{1}}\y_{0}\rv\rV_{1}.\\
		\end{split}
	\end{eqnarray*}
	
	We then present the following lemma. 
	\begin{lemma}\label{A4}
		For any $\pmb{z}_{1},\pmb{z}_{2}$ and 
		$\pmb{a}_{i}\stackrel{i.i.d}{\sim}\mathcal{N}\left(0,\pmb{I}_{n}\right)$, we have
		\begin{eqnarray}
			\mathbb{P}\left(\frac{1}{m}\lV\lv \mA\pmb{z}_{1}\rv-\lv \mA\pmb{z}_{2}\rv\rV_{1}\ge\sqrt{\frac{2}{\pi}}\left(1+\epsilon\right)\lV\pmb{z}_{1}-\pmb{z}_{2}\rV_{2}\right)\le2\exp\left(-c_{1}m\epsilon^{2}/K^{2}_{1}\right),
		\end{eqnarray}
		where $K_{1}=\sqrt{8/3}$ and $c_{1}$ is a positive numerical constant.
	\end{lemma}
	\begin{proof}
		It can be seen that  $\big|\lv \langle\pmb{a},\pmb{z}_{1}\rangle\rv-\lv \langle\pmb{a},\pmb{z}_{2}\rangle\rv\big|\le\lv \langle\pmb{a},\pmb{z}_{1}- \pmb{z}_{2}\rangle)\rv$ and $\mathbb{E}\left[\lv \langle\pmb{a},\pmb{z}_{1}- \pmb{z}_{2}\rangle)\rv\right]=\sqrt{\frac{2}{\pi}}\lV\pmb{z}_{1}- \pmb{z}_{2}\rV_{2}$.				
Furthermore, random variable $\lv \langle\pmb{a},\pmb{z}/\lV\pmb{z}\rV_{2}\rangle\rv$ has sub-Guassian norm $\sqrt{8/3}$. Then by Hoeffding inequality in \cite{vershynin2018high}, we have reached the conclusion.
		\end{proof}
Jointing the above two inequalities and Lemma \ref{A4}, we obtain that
		\begin{eqnarray*} 
		\begin{split}
			&\min_{\mS \subset[m],|\mS|\leq s m} \lV\lv\mA_{{\mS}^{c}}\x\rv-\lv\mA_{{\mS}^{c}}\y\rv\rV_{1}-\lV\lv\mA_{{\mS}}\x\rv-\lv\mA_{{\mS}}\y\rv\rV_{1}\\
			&\ge\min_{\mS \subset[m],|\mS|\leq s m}\lV\lv\mA_{{\mS}^{c}}\x_{0}\rv-\lv\mA_{{\mS}^{c}}\y_{0}\rv\rV_{1}-\lV\lv\mA_{{\mS}}\x_{0}\rv-\lv\mA_{{\mS}}\y_{0}\rv\rV_{1}			-\lV\lv\mA\x\rv-\lv\mA\x_{0}\rv-\lv\mA\y\rv+\lv\mA\y_{0}\rv\rV_{1}\\	
			&\ge m\left[\mathcal{M}\left(s+\xi\right)-2l_{1}\xi\right]\cdot\mathrm{dist}\left(\x_{0},\y_{0}\right)-
			\lV\lv\mA\x\rv-\lv\mA\x_{0}\rv\rV_{1}-\lV\lv\mA\y\rv-\lv\mA\y_{0}\rv\rV_{1}\\		
			&\ge m\left[\mathcal{M}\left(s+\xi\right)-2l_{1}\xi\right]\cdot\left(\mathrm{dist}_{1}\left(\x,\y\right)-2\delta\right)-\sqrt{2/\pi}\left(1+\epsilon\right)\left(\lV\x-\x_{0}\rV_{2}+\lV\y-\y_{0}\rV_{2}\right)\\
			&\ge m\left[\mathcal{M}\left(s+\xi\right)-2l_{1}\xi\right]\cdot\mathrm{dist}_{1}\left(\x,\y\right)-\left[2\sqrt{2/\pi}\left(1+\epsilon\right)+2\mathcal{M}\left(0\right)\right]m\delta\\
			&\ge m\left[\mathcal{M}\left(s+\xi\right)-2l_{0}\xi\right]\cdot\mathrm{dist}_{1}\left(\x,\y\right).
			\end{split}
		\end{eqnarray*}
		In the forth line we use Lemma \ref{A4}  and the fact that $\lV\x_{0}-\y_{0}\rV_{2}\ge\lV\x-\y\rV_{2}-2\delta$.
		We can set $l_{0}\textgreater l_{1}$ such that $\left[2\sqrt{2/\pi}\left(1+\epsilon\right)+2\mathcal{M}\left(0\right)\right]\delta+l_{1}\xi\cdot\mathrm{dist}_{1}\left(\x,\y\right)=l_{0}\xi\cdot\mathrm{dist}_{1}\left(\x,\y\right)$ (thus $\delta\ge\tilde{c}\xi$) in the last step as $\mathrm{dist}_{1}\left(\x,\y\right)=\lV\x-\y\rV_{2}\ge\lV\x+\y\rV_{2}/5\ge1/5$.
			
\textbf{Case 2}:  We assume that $5\lV\x-\y\rV_{2}\le\lV\x+\y\rV_{2}$.  Due to $\x,\y\in\mathcal{K}$, we have $\x-\y\in\mathcal{K}-\mathcal{K}$. By standardization, we let $\x-\y\in\left(\mathcal{K}-\mathcal{K}\right)\cap\mathcal{S}^{n}$. Then let $\mathcal{K}^{-}_{\delta }$ be the $\delta$-net of $\text{cone}\left(\mathcal{K}-\mathcal{K}\right)\cap \mathbb{S}^{n-1}$. Now for all $\pmb{u}:=\x-\y\in\left(\mathcal{K}-\mathcal{K}\right)\cap\mathcal{S}^{n}$, their exist $\pmb{u}_{0}:=\x_{0}-\y_{0}\in\mathcal{K}^{-}_{\delta}$, such that $\lV\pmb{u}-\pmb{u}_{0}\rV_{2}\le\delta$. 

For $\rho=\langle\x/\lV\x\rV_{2},\y/\lV\y\rV_{2}\rangle$, set $\mathcal{K}^{\rho}_{\delta}$ be the $\delta$-net of $\rho\in\left[0,1\right]$. Thus $\mathcal{N}\left(\mathcal{K}^{\rho}_{\delta },\delta\right)\le\mathcal{N}\left(\mathcal{K}^{1}_{\delta },\delta\right)\times\mathcal{N}\left(\mathcal{K}^{2}_{\delta },\delta\right)$. 
Due to $\lV\x+\y\rV_{2}\ge5\lV\x-\y\rV_{2}$, we have $\rho\textgreater0$. There exists $\rho_{0}\in\mathcal{K}^{\rho}_{\delta }$, such that $\lv\rho-\rho_{0}\rv\le\delta$, then for small enough $\delta$, we have both $\rho,\rho_{0}\textgreater0$. Then we can get
\begin{eqnarray*} 
		\begin{split}
		\mathbb{E}\big|\lv\langle\pmb{a},\x\rangle\rv-\lv\langle\pmb{a},\x_{0}\rangle\rv-\lv\langle\pmb{a},\y\rangle\rv+\lv\langle\pmb{a},\y_{0}\rangle\rv\big|
		&=\mathbb{E}\big|\lv\langle\pmb{a},\x-\y\rangle\rv-\lv\langle\pmb{a},\x_{0}-\y_{0}\rangle\rv\big|\\
		&=\mathbb{E}\big|\lv\langle\pmb{a},\pmb{u}\rangle\rv-\lv\langle\pmb{a},\pmb{u}_{0}\rangle\rv\big|
		\le\mathbb{E} \lv\langle \pmb{a},\pmb{u}-\pmb{u}_{0}\rangle\rv \le \sqrt{\frac{2}{\pi}}\delta.
		\end{split}
		\end{eqnarray*}
		Due to Lemma \ref{A4}, we have the following concentration inequality
\begin{eqnarray}
			\mathbb{P}\left(\frac{1}{m}\lV\lv \mA\pmb{x}\rv-\lv \mA\pmb{y}\rv-\lv \mA\pmb{x}_{0}\rv+\lv \mA\pmb{y}_{0}\rv\rV_{1}\ge \sqrt{\frac{2}{\pi}}\delta\left(1+\epsilon\right)\right)\le2\exp\left(-c_{1}m\epsilon^{2}\right).
		\end{eqnarray}
 We then obtained that
		\begin{eqnarray*} 
		\begin{split}
			&\min_{\mS \subset[m],|\mS|\leq s m} \lV\lv\mA_{{\mS}^{c}}\x\rv-\lv\mA_{{\mS}^{c}}\y\rv\rV_{1}-\lV\lv\mA_{{\mS}}\x\rv-\lv\mA_{{\mS}}\y\rv\rV_{1}\\
			&\ge m\left[\mathcal{M}\left(s+\xi\right)-2l_{1}\xi\right]\cdot\mathrm{dist}_{1}\left(\x_{0},\y_{0}\right)-
			\lV\lv\mA\x\rv-\lv\mA\x_{0}\rv-\lv\mA\y\rv+\lv\mA\y_{0}\rv\rV_{1}\\		
			&\ge m\left[\mathcal{M}\left(s+\xi\right)-2l_{1}\xi\right]\cdot\left(\mathrm{dist}_{1}\left(\x,\y\right)-\delta\right)-\sqrt{\frac{2}{\pi}}\delta\left(1+\epsilon\right)m\cdot\mathrm{dist}_{1}\left(\x,\y\right)\\
			&\ge m\left[\mathcal{M}\left(s+\xi\right)-2l_{1}\xi\right]\cdot\mathrm{dist}_{1}\left(\x,\y\right)-\left[\sqrt{\frac{2}{\pi}}\left(1+\epsilon\right)+\mathcal{M}\left(0\right)\right]m\delta\cdot\mathrm{dist}_{1}\left(\x,\y\right)		\\
			&= m\left[\mathcal{M}\left(s+\xi\right)-2l_{0}\xi\right]\cdot\mathrm{dist}_{1}\left(\x,\y\right).
			\end{split}
		\end{eqnarray*}
				
  	\textbf{Case 3}:  We assume that $\lV\x+\y\rV_{2}\le\lV\x-\y\rV_{2}$.  Then $\x+\y\in\mathcal{K}+\mathcal{K}$, we let $\x+\y\in\left(\mathcal{K}+\mathcal{K}\right)\cap\mathcal{S}^{n}$. Then let $\mathcal{K}^{+}_{\delta }$ be the $\delta$-net of $\text{cone}\left(\mathcal{K}+\mathcal{K}\right)\cap \mathbb{S}^{n-1}$. The remain classification and arguments are similar to \textbf{Case 1} and \textbf{Case 2}.

By properties $w\left(\mathcal{K}+\mathcal{K}\right)=2w\left(\mathcal{K}\right), w\left(\mathcal{K}-\mathcal{K}\right)=\frac{1}{2}w\left(\mathcal{K}\right)$, we have 
	\begin{equation*}
	\begin{split}
	&\max\{C_{1}w^2\left(\text{cone}(\mathcal{K})\cap B_2^n\right),C_{2}w^{2}\left(\text{cone}\left(\mathcal{K}-\mathcal{K}\right)\cap\mathbb{S}^{n-1}\right),C_{3}w^{2}\left(\text{cone}\left(\mathcal{K}+\mathcal{K}\right)\cap\mathbb{S}^{n-1}\right)\}\\
	&=C_{4}w^{2}\left(\text{cone}\left(\mathcal{K}\right)\cap\mathbb{S}^{n-1}\right).
	\end{split}
\end{equation*}
	Then by Sudakov’s minoration inequality in (\ref{Sudakov}), for any $\delta\textgreater0$,
	we can get 
	\begin{equation*}
	\begin{split}
	&\max\{\log\mathcal{N}\left(\mathcal{K}^{1}_{\delta },\delta\right),\log\mathcal{N}\left(\mathcal{K}^{2}_{\delta },\delta\right),\log\mathcal{N}\left(\mathcal{K}^{-}_{\delta },\delta\right),\log\mathcal{N}\left(\mathcal{K}^{+} 
	_{\delta },\delta\right)\}\\
&\le C_{4}w^{2}\left(\text{cone}\left(\mathcal{K}\right)\cap\mathbb{S}^{n-1}\right)/\delta^{2}.
\end{split}
	\end{equation*}

		Finally, we can finish the proof with probability at least 
		\begin{eqnarray*} 
			&&1-
			C_{5}\left(\mathcal{N}\left(\mathcal{K}^{1}_{\delta },\delta\right)\times\mathcal{N}\left(\mathcal{K}^{2}_{\delta },\delta\right)\times\mathcal{N}\left(\mathcal{K}^{-}_{\delta },\delta\right)\times\mathcal{N}\left(\mathcal{K}^{+}_{\delta },\delta\right)\right)\left(\mathcal{O}\left(e^{-cm\xi^{2}}\right)+2e^{-c_{1}m\delta^{2}}\right)\\
			&\ge&1-C_{6}\exp\left[w^{2}\left(\text{cone}\left(\mathcal{K}\right)\cap\mathbb{S}^{n-1}\right)/\delta^{2}\right]\left(\mathcal{O}\left(e^{-cm\xi^{2}}\right)+2e^{-c_{1}m\delta^{2}}\right)\\
			&\ge&1-C_{7}\exp\left[w^{2}\left(\text{cone}\left(\mathcal{K}\right)\cap\mathbb{S}^{n-1}\right)/\xi^{2}\right]\cdot\mathcal{O}\left(e^{-c_{2}m\xi^{2}}\right)
			=1-\mathcal{O}\left(e^{-c_{0}m\xi^{2}}\right),
		\end{eqnarray*}
		as we have provided $m \gtrsim \xi ^{-4}\cdot w^{2}\left(\text{cone}(\mathcal{K})\cap\mathbb{S}^{n-1}\right)$.
		
	We then present the proof of Corollary \ref{coro1}. We first set $s=0$, then we have 
		\begin{eqnarray*} 
			\frac{1}{m}\lV \lv\mA \x\rv -\lv\mA \y\rv \rV_1\ge \left[\mathcal{M}\left(\xi\right)-l_{0}\xi\right]	\cdot\mathrm{dist}_{1}\left(\x,\y\right).
		\end{eqnarray*}
		By Proposition \ref{K}, 
		we can assume $\mathcal{M}\left(s\right)$ has Lipschitz constant $L$. 
		Thus $\lv\mathcal{M}\left(0\right)-\mathcal{M}\left(\xi\right)\rv\le L\xi$ and we set $\tilde{l}=L+l_{0}$ to get the left-hand inequality of (\ref{co1}).
	We set $s=1-\xi$ to get
	\begin{eqnarray*} 
			\frac{1}{m}\lV \lv\mA \x\rv -\lv\mA \y\rv \rV_1&\le& \left[-\mathcal{M}\left(1\right)+l_{0}\xi\right]\cdot\mathrm{dist}_{1}\left(\x,\y\right)+
			\frac{2}{m}\min_{\left\{\mS \subset\left[m\right],  |\mS|\leq \left(1-\xi\right)m\right\} }
			\lV \lv\mA_{{\mS}^{c}} \x\rv -\lv\mA_{{\mS}^{c}} \y\rv \rV_1\\
			&\le&\left(\sqrt{\frac{2}{\pi}}+l_{0}\xi\right)\cdot\mathrm{dist}_{1}\left(\x,\y\right)+
			\frac{2}{m}\xi
			\lV \lv\mA\x\rv -\lv\mA \y\rv \rV_1\\
			&\le&\left(\sqrt{\frac{2}{\pi}}+l_{0}\xi\right)\cdot\mathrm{dist}_{1}\left(\x,\y\right)+
			4\xi\lV\x-\y\rV_{2}	\le	\left[\sqrt{\frac{2}{\pi}}+\left(l_{0}+4\right)\xi\right]\cdot\mathrm{dist}_{1}\left(\x,\y\right).
			\end{eqnarray*}
In the last second inequality we use Lemma \ref{A4}. Actually the right-hand inequality of (\ref{co1}) can be obtained concisely using Lemma \ref{A4}, but we omit it here.
	
		\subsection{Nonlinear ROB Condition for Intensity Measurement}\label{nonlinear_2}
		It can be checked that 
		\begin{eqnarray*}
			\frac{\lV\lv\mA\pmb{x}\rv^{2}-\lv\mA\pmb{y}\rv^{2}\rV_{1}}{\mathrm{dist}_{2}\left(\x,\y\right)}
			&=& \sum_{i=1}^{m}\lvert \langle\pmb{a}_{i}, \frac{\x-\y}{\lV\x-\y\rV_{2}}\rangle \cdot\langle\pmb{a}_{i}, \frac{\x+\y}{\lV\x+\y\rV_{2}}\rangle \rvert\\
			&=& \sum_{i=1}^{m}\lvert \langle\pmb{a}_{i}, \pmb{u}\rangle  \cdot\langle\pmb{a}_{i}, \pmb{v}\rangle \rvert :=\sum_{i=1}^{m}\lvert X\cdot Y\rvert,
		\end{eqnarray*}	
		where we set $\pmb{u}=\frac{\x-\y}{\lV\x-\y\rV_{2}}\in\mathbb{S}^{n-1},\pmb{v}=\frac{\x+\y}{\lV\x+\y\rV_{2}}\in\mathbb{S}^{n-1}$ and $X=\langle\pmb{a}_{i}, \pmb{u}\rangle\sim \mathcal{N}\left(0,1\right), Y=\langle\pmb{a}_{i}, \pmb{v}\rangle\sim \mathcal{N}\left(0,1\right)$.
		The correlation coefficient between random variables $X$ and $Y$ is 
		$\rho:=\langle\pmb{u}, \pmb{v}\rangle$,
		then by Lemma 2 in \cite{huang2022outlier}, the PDF of $\left|Z_{\rho}\right|:=\left|X\cdot Y\right|$ is
		\begin{equation} \label{|xy|}
			f_{\rho}\left(z\right) = 
			\begin{cases} 
				\frac{1}{\sqrt{2\pi z}}e^{-z/2} &\rho=\pm 1 \\ 
				\frac{2}{\pi\sqrt{1-\rho^2}} \cosh \left(\frac{\rho z}{1-\rho^2}\right) K_0\left( \frac{z}{1-\rho^2} \right) &-1<\rho<1,
			\end{cases} 
		\end{equation}
		where $K_0(\cdot)$ denotes the modified Bessel function of the second kind of order zero. For completeness, we attach the proof of (\ref{|xy|}) to Appendix \ref{RV}. 
		Then we set the CDF of $|Z_{\rho}|$ as
		\begin{eqnarray*}
			F_{\rho}(t) = \int_{0}^{t} f_{\rho}(z) dz.
		\end{eqnarray*}
		And we define the balance function for $\lv Z_{\rho}\rv$:
		\begin{eqnarray*}
			J\left(\rho, s\right) = J_{1}\left(\rho, s\right)-J_{2}\left(\rho, s\right):=\left[ \int_0^{F_{\rho}^{-1}\left(1-s\right)} - \int_{F_{\rho}^{-1}\left(1-s\right)}^{+\infty} \right] z f_{\rho}(z)dz.
		\end{eqnarray*}
		We can only consider $\rho\ge 0$, thus the minimum balance function here is
		\begin{eqnarray}
			\mathcal{J}\left(s\right)=\min_{\rho \in\left[0,1\right]}\left[\int_{0 }^{F_{\rho}^{-1}\left(1-s\right)}- \int_{F_{\rho}^{-1}\left(1-s\right) }^{+\infty} \right]z f_{\rho }\left(z\right)dz.
		\end{eqnarray}
		Now we present the following theorem.
	\begin{theorem}\label{k=2}
		Let 
		$\mA\x=\{\langle\pmb{a}_{1},\x\rangle,\cdots,\langle\pmb{a}_{m},\x\rangle\}$ where $\pmb{a}_{i}\stackrel{i.i.d}{\sim}\mathcal{N}\left(0,\pmb{I}_{n}\right)$.
		Assume that $0 \textless s\textless s^{*,2}\approx0.1185$ and $c_0$ is a positive numerical constant corresponding to $l_{0}$. 
		If
		\begin{eqnarray*}
			m \gtrsim \xi ^{-4}\cdot w^2\left(\text{cone}\left(\mathcal{K}\right)\cap \mathbb{S}^{n-1}\right),
		\end{eqnarray*}
		then for any $s<s^{*,2}$,  $\mA$ satisfies the $s$-nonlinear ROB condition for all $\x,\y\in \mathcal{K}$:
		\begin{equation}
		\begin{split}
			\frac{1}{m}\min_{\left\{\mS \subset\left[m\right],  |\mS|\leq s m\right\} }
			\Bigl[\lV \lv\mA_{{\mS}^{c}} \x\rv^{2} -\lv\mA_{{\mS}^{c}} \y\rv^{2} \rV_1&-\lV \lv\mA_{\mS} \x\rv^{2} -\lv\mA_{\mS} \y\rv^{2} \rV_1 \Bigr]\\
			&\ge\left[\mathcal{J}\left(s+\xi \right)-l_{0}\xi\right]\cdot\mathrm{dist}_{2}\left(\x,\y\right)
			\end{split}
		\end{equation}
		with probability at least $1- \mathcal{O}\left(e^{-c_{0}m\xi^{2}}\right)$.
	\end{theorem}
	\begin{remark}
	The distinction between nonlinear ROB condition and ROB condition in intensity measurement, as discussed in \cite{huang2022outlier}, lies in the specific ranges of consideration. The former limits the consideration range to $\pmb{u}=\frac{\x-\y}{\lV\x-\y\rV_{2}},\pmb{v}=\frac{\x+\y}{\lV\x+\y\rV_{2}}\in\mathbb{S}^{n-1}$, while the latter considers the tangent space $T=\{ \x\y^{*}+\y\x^{*}:\x,\y\in\mathbb{R}^{n}\}$. However, both conditions ultimately lead to the same threshold value of 0.1185.
		\end{remark}
		Similar to the Corollary \ref{coro1}, if we set $s=0$ and $s=1-\xi$ in the nonlinear ROB condition for intensity measurement respectively, wo have the following corollary.
	\begin{corollary}\label{coro2}
		When $\mathcal{K}=\mathbb{R}^{n}$ and $m \gtrsim \xi ^{-4}\cdot n$  
		 with probability at least $1- \mathcal{O}\left(e^{-c_{0}m\xi^{2}}\right)$, we have 
		\begin{eqnarray*}
					\left(\frac{2}{\pi}-\tilde{l}\xi\right)
			\cdot\mathrm{dist}_{2}\left(\x,\y\right)
			\le\frac{1}{m}\lV \lv\mA \x\rv^{2} -\lv\mA \y\rv^{2} \rV_1
			\le \left(  1+   \tilde{l}\xi\   \right)\cdot\mathrm{dist}_{2}\left(\x,\y\right).
		\end{eqnarray*}
		\end{corollary}
	Therefore, our nonlinear ROB condition can be seen as a version with outliers in the stability of phase retrieval in \cite{eldar2014phase}. 

	\subsection{Proof of Theorem \ref{k=2} and Corollary \ref{coro2}}
	In a manner analogous to the demonstration of Theorem \ref{k=1}, the proof of Theorem 6 will be partitioned into three steps.
	For simplicity and convenience of reading, we attach the proofs of \textbf{Step 1} and \textbf{Step 2} to Appendix \ref{partial} and 
	due to the similarity between the proof of Corollary \ref{coro2} and Corollary \ref{coro1}, we omit it.
	
	{\bf Step 1: Properties of $\mathcal{J}\left(s\right)$.}
This step provides that $\mathcal{J}\left(s\right)$ is a well-defined function and has a unique zero $s^{*,2}\approx0.1185$.
	\begin{proposition}\label{J}
		The minimum balance function $\mathcal{J}\left(s\right)$ satisfies:
		\begin{itemize}
			\item[$\mathrm{(a)}$]$\mathcal{J}\left(s\right)$ is a strictly monotonically decreasing continuous function on interval $\left[0,1\right]$;
			\item[$\mathrm{(b)}$]$\mathcal{J}\left(s\right)$ has a unique zero $s^{*,2}$, that is $\mathcal{J}\left(s^{*,2}\right)= 0$. 
		\end{itemize}
	\end{proposition}
	
		{\bf Step 2: Fixed vectors description.} 
	\begin{lemma}\label{B}
		Fix $\x,\y\in \mathcal{K}$ and let $\rho=\langle\frac{\x-\y}{\lV\x-\y\rV_{2}},\frac{\x+\y}{\lV\x+\y\rV_{2}}\rangle$. 
		For $0 \textless s \textless s^{*,2}$, there exists a positive numerical constant $c_{1}$ corresponding to $l_{1}$ such that with probability at least $1-\mathcal{O}\left(e^{-c_{1}m\xi^{2}}\right)$,	
		\begin{eqnarray}
			\min_{\mS \subset[m],  |\mS|\leq s m}\frac{\lV\lv\mA_{{\mS}^{c}}\pmb{x}\rv^{2}-\lv\mA_{{\mS}^{c}}\pmb{y}\rv^{2}\rV_{1}}{\mathrm{dist}_{2}\left(\x,\y\right)}
			&\geq& m\left[J_{1}(\rho, s+\xi)-l_{1}\xi\right],\label{B5}\\	
			\max_{\mS \subset[m],  |\mS|\leq s m}\frac{\lV\lv\mA_{\mS}\pmb{x}\rv^{2}-\lv\mA_{\mS}\pmb{y}\rv^{2}\rV_{1}}{\mathrm{dist}_{2}\left(\x,\y\right)}
			&\leq& m \left[J_{2}(\rho, s+\xi)+l_{1}\xi \right].\label{B6}
		\end{eqnarray}
	\end{lemma}
		
	{\bf Step 3: Uniform  argument.}
	By Lemma \ref{B}, we have that for fixed $\x_{0},\y_{0}\in \mathcal{K}$ with $\rho_{0}=\langle \frac{\x_{0}-\y_{0}}{\lV\x_{0}-\y_{0}\rV_{2}},\frac{\x_{0}+\y_{0}}{\lV\x_{0}+\y_{0}\rV_{2}} \rangle:=\langle \pmb{u}_{0},\pmb{v}_{0}\rangle$, with probability at least $1 - \mathcal{O}\left(e^{- c_{1} m\xi^2}\right)$,
	\begin{equation*}\label{B1}
	\begin{split}
		&\min_{\mS \subset[m],  |\mS|\leq s m} \frac{\lV\lv\mA_{{\mS}^{c}}\x_{0}\rv^{2}-\lv\mA_{{\mS}^{c}}\y_{0}\rv^{2}\rV_{1}-\lV\lv\mA_{{\mS}}\x_{0}\rv^{2}-\lv\mA_{{\mS}}\y_{0}\rv^{2}\rV_{1}}{\mathrm{dist}_{2}\left(\x_{0},\y_{0}\right)}\\
		&= \min_{\mS \subset[m],  |\mS|\leq s m} \left[\sum_{i\in{\mS}^{c}} \lv \langle\pmb{a}_{i}, \pmb{u}_{0}\rangle \cdot\langle\pmb{a}_{i}, \pmb{v}_{0}\rangle \rv-\sum_{i\in\mS} \lv \langle\pmb{a}_{i}, \pmb{u}_{0}\rangle \cdot\langle\pmb{a}_{i}, \pmb{v}_{0}\rangle \rv\right]\\
		&\ge m\left[J_{1}\left(\rho_{0},s+\xi\right)-J_{2}\left(\rho_{0},s+\xi\right)-2l_{1}\xi\right]\\
		&\ge m\left[\mathcal{J}\left(s+\xi\right)-2l_{1}\xi\right].
		\end{split}
	\end{equation*}
	
	As we can see $\pmb{u}_{0}\in \left(\mathcal{K}-\mathcal{K}\right)\cap\mathbb{S}^{n-1}$ and $\pmb{v}_{0}\in \left(\mathcal{K}+\mathcal{K}\right)\cap\mathbb{S}^{n-1}$.
	Let $\mathcal{K}^{-}_{\delta },\mathcal{K}^{+}_{\delta }$ be the $\delta$-net of $\text{cone}\left(\mathcal{K}-\mathcal{K}\right)\cap\mathbb{S}^{n-1}$ and $\text{cone}\left(\mathcal{K}+\mathcal{K}\right)\cap\mathbb{S}^{n-1}$. 

	Now for all $\pmb{u}\in\left(\mathcal{K}-\mathcal{K}\right)\cap\mathcal{S}^{n},\pmb{v}\in \left(\mathcal{K}+\mathcal{K}\right)\cap \mathcal{S}^{n}$, there exist $\pmb{u}_{0}\in\mathcal{K}^{-}_{\delta },\pmb{v}_{0}\in\mathcal{K}^{+}_{\delta }$, such that $\lV\pmb{u}-\pmb{u}_{0}\rV_{2}\le\delta, \lV\pmb{v}-\pmb{v}_{0}\rV_{2}\le\delta$. Then set $\mS_{1}$ be the index set of the largest $s$-fraction of $\mS_{\pmb{u},\pmb{v}}=\{\langle\pmb{a}_{i}, \pmb{u}\rangle \cdot\langle\pmb{a}, \pmb{v}\rangle  | i\in[m] 
	\}$ in absolute value.  
	Similarly, $\mS_{2}$ be the subset of $\mS_{\pmb{u}_{0},\pmb{v}_{0}}$, collecting the index of the largest $s$-fraction of $\mS_{\pmb{u}_{0},\pmb{v}_{0}}$ in absolute 
	value.
	Then we have
	\begin{equation*}
	\begin{split}
		\min_{\mS \subset[m],  |\mS|\leq s m}
		&\frac{\lV\lv\mA_{{\mS}^{c}}\x\rv^{2}-\lv\mA_{{\mS}^{c}}\y\rv^{2}\rV_{1}}{\mathrm{dist}_{2}\left(\x,\y\right)}
		=\frac{\lV\lv\mA_{{\mS}_{1}^{c}}\x\rv^{2}-\lv\mA_{{\mS}_{1}^{c}}\y\rv^{2}\rV_{1}}{\mathrm{dist}_{2}\left(\x,\y\right)}
		=\sum_{i\in{\mS}_{1}^{c}} \lv \langle\pmb{a}_{i}, \pmb{u}\rangle \cdot\langle\pmb{a}_{i}, \pmb{v}\rangle \rv\\
		&\ge\sum_{i\in{\mS}_{2}^{c}} \lv \langle\pmb{a}_{i}, \pmb{u}_{0}\rangle \cdot\langle\pmb{a}_{i}, \pmb{v}_{0}\rangle \rv-
		\sum_{i\in{\mS}_{1}^{c}} \lv \langle\pmb{a}_{i}, \pmb{u}\rangle \cdot\langle\pmb{a}_{i}, \pmb{v}\rangle- \langle\pmb{a}_{i}, \pmb{u}_{0}\rangle \cdot\langle\pmb{a}_{i}, 
		\pmb{v}_{0}\rangle\rv\\
		&=\frac{\lV\lv\mA_{{\mS}_{2}^{c}}\x_{0}\rv^{2}-\lv\mA_{{\mS}_{2}^{c}}\y_{0}\rv^{2}\rV_{1}}{\mathrm{dist}_{2}\left(\x_{0},\y_{0}\right)}
		-\sum_{i\in{\mS}_{1}^{c}} \lv \langle\pmb{a}_{i}\pmb{a}^{*}_{i}, \pmb{u}\pmb{v}^{*}-\pmb{u}_{0}\pmb{v}_{0}^{*}\rangle\rv\\
		&=	\min_{\mS \subset[m],  |\mS|\leq s m} \frac{\lV\lv\mA_{{\mS}^{c}}\x_{0}\rv^{2}-\lv\mA_{{\mS}^{c}}\y_{0}\rv^{2}\rV_{1}}{\mathrm{dist}_{2}\left(\x_{0},\y_{0}\right)}
		-\sum_{i\in{\mS}_{1}^{c}} \lv \langle\pmb{a}_{i}\pmb{a}^{*}_{i}, \pmb{u}\pmb{v}^{*}-\pmb{u}_{0}\pmb{v}_{0}^{*}\rangle\rv.
		\end{split}
	\end{equation*}
	Similarly, we can get
	\begin{equation*}
	\begin{split}
		&\max_{\mS \subset[m],  |\mS|\leq s m} \frac{\lV\lv\mA_{{\mS}}\x\rv^{2}-\lv\mA_{{\mS}}\y\rv^{2}\rV_{1}}{\mathrm{dist}_{2}\left(\x,\y\right)}\\
		&\le\max_{\mS \subset[m],  |\mS|\leq s m} \frac{\lV\lv\mA_{{\mS}}\x_{0}\rv^{2}-\lv\mA_{{\mS}}\y_{0}\rv^{2}\rV_{1}}{\mathrm{dist}_{2}\left(\x_{0},\y_{0}\right)}
		-\sum_{i\in{\mS}_{1}} \lv \langle\pmb{a}_{i}\pmb{a}^{*}_{i}, \pmb{u}\pmb{v}^{*}-\pmb{u}_{0}\pmb{v}_{0}^{*}\rangle\rv.
		\end{split}
	\end{equation*}
	Before presenting our final conclusion, we provide the following lemma. 
	\begin{lemma}\label{B4}
		For any $\pmb{u},\pmb{v}\in\mathbb{S}^{n-1}$ and $\pmb{a}_{i}\stackrel{i.i.d}{\sim}\mathcal{N}\left(0,\pmb{I}_{n}\right)$, we have
		\begin{eqnarray}
			\mathbb{P}\left(\frac{1}{m}\sum_{i=1}^{m}\lv \langle\pmb{a}_{i}\pmb{a}^{*}_{i}, \pmb{u}\pmb{v}^{*}\rangle\rv\ge1+\epsilon\right)\le2\exp\left[-c_{1}
			m\min\left(\frac{\epsilon^{2}}{K_{1}^{2}},\frac{\epsilon}{K_{1}}\right)\right],
		\end{eqnarray}
		where $K_{1}=\sqrt{8/3}$ and $c_{1}$ is a positive numerical constant.
	\end{lemma}
	\begin{proof}
		It can be seen that $U:=\langle\pmb{a}, \pmb{u}\rangle\sim\mathcal{N}\left(0,1\right),V:=\langle\pmb{a}, \pmb{u}\rangle\sim\mathcal{N}\left(0,1\right)$.
	Furthermore, $\| U\cdot V\|_{\psi_1}\le  \| U\|^{2}_{\psi_2} =\| V\|^{2}_{\psi_2} =8/3$
	and by the proof Proposition \ref{J}.(a) in Appendix \ref{partial}, $\mathbb{E}|U\cdot V|\le 1$. Thus, by Bernstein inequality in \cite{vershynin2018high}, we have reached the conclusion as
	$\langle\pmb{a}\pmb{a}^{*}, \pmb{u}\pmb{v}^{*}\rangle=\langle\pmb{a}, \pmb{u}\rangle\langle\pmb{a}, \pmb{v}\rangle=U\cdot V$.
	\end{proof}
	From the above two inequalities and Lemma \ref{B4}, a unified form can be obtained that
	\begin{equation*}
	\begin{split} 
		&\min_{\mS \subset[m],|\mS|\leq s m} \frac{\lV\lv\mA_{{\mS}^{c}}\x\rv^{2}-\lv\mA_{{\mS}^{c}}\y\rv^{2}\rV_{1}-\lV\lv\mA_{{\mS}}\x\rv^{2}-\lv\mA_{{\mS}}\y\rv^{2}\rV_{1}}
		{\mathrm{dist}_{2}\left(\x,\y\right)}\\
		&\ge\min_{\mS \subset[m],|\mS|\leq s m} \frac{\lV\lv\mA_{{\mS}^{c}}\x_{0}\rv^{2}-\lv\mA_{{\mS}^{c}}\y_{0}\rv^{2}\rV_{1}-\lV\lv\mA_{{\mS}}\x_{0}\rv^{2}-\lv\mA_{{\mS}}
		\y_{0}\rv^{2}\rV_{1}}{\mathrm{dist}_{2}\left(\x_{0},\y_{0}\right)}
		-\sum_{i=1}^{m} \lv \langle\pmb{a}_{i}\pmb{a}^{*}_{i}, \pmb{u}\pmb{v}^{*}-\pmb{u}_{0}\pmb{v}_{0}^{*}\rangle\rv\\
		&\ge m\left[\mathcal{J}\left(s+\xi\right)-2l_{1}\xi\right]-
		\sum_{i=1}^{m} \left[\lv \langle\pmb{a}_{i}\pmb{a}^{*}_{i}, \left(\pmb{u}-\pmb{u}_{0}\right)\pmb{v}^{*}\rangle\rv+\lv \langle\pmb{a}_{i}\pmb{a}^{*}_{i}, \pmb{u}_{0}
		\left(\pmb{v}-\pmb{v}_{0}\right)^{*}\rangle\rv\right]\\	
		&\ge m\left[\mathcal{J}\left(s+\xi\right)-2l_{1}\xi\right]-2\delta\left(1+\epsilon\right)=m\left[\mathcal{J}\left(s+\xi\right)-2l_{0}\xi\right].
		\end{split}
	\end{equation*}
	Here, in the penultimate inequality, we use Lemma \ref{B4} for $\lV\pmb{u}-\pmb{u}_{0}\rV\le\delta,\lV\pmb{v}-\pmb{v}_{0}\rV\le \delta$ and in the last inequality we set $l_{0}\xi=l_{1}\xi+\delta\left(1+\epsilon\right)$.
	
	Finally, we  can finish the proof with probability at least 
	\begin{eqnarray*} 
	\begin{split}
		&1-\left(\mathcal{N}\left(\mathcal{K}^{-}_{\delta},\delta\right)\times\mathcal{N}\left(\mathcal{K}^{+}_{\delta},\delta\right)\right)\left(\mathcal{O}\left(e^{-cm\xi^{2}}\right)+2e^{-c_{1}m\delta^{2}}\right)\\
		\ge&1-C_{0} \exp\left[w^{2}\left(\text{cone}\left(\mathcal{K}\right)\cap\mathbb{S}^{n-1}\right)/\xi^{2}\right]\cdot \mathcal{O}\left(e^{-c_{2}m\xi^{2}}\right)
		=1-\mathcal{O}\left(e^{-c_{0}m\xi^{2}}\right),
		\end{split}
	\end{eqnarray*}
	provided $m\gtrsim \xi ^{-4}\cdot w^{2}\left(\text{cone}(\mathcal{K})\cap\mathbb{S}^{n-1}\right)$.

\section{Proof of Main Results}\label{main results}
	
	\subsection{Recovery Guarantee for Amplitude Measurement }\label{proof amp}
	The proof of Theorem \ref{theorem1} is based on the adversarial sparse outlier separation condition and the nonlinear ROB  condition for amplitude measurement.
	\begin{proof}[Proof of Theorem \ref{theorem1}]
	By Lemma \ref{ASOSC1} and Theorem \ref{k=1}, with probability exceeding $1-\mathcal{O}\left(e^{-cm\xi^{2}}\right)$,
	\begin{eqnarray*}
	\mathrm{dist}_{1}\left(\x_{\star},\x_{0}\right)\leq \frac{2}{\mathcal{M}\left(s+\xi\right)-l_{0}\xi}\frac{\lV\pmb{\omega}\rV_1}{m}.
	\end{eqnarray*}
Since $\mathcal{M}\left(s\right)$ is monotonically decreasing, we have $\mathcal{M}\left(s+\xi\right)\textgreater \mathcal{M}\left(s^{*,1}-\xi\right)\textgreater 0$. Based on the proof of Theorem \ref{k=1}, the constant $l_{0}$ can be chosen small enough such that $ \mathcal{M}\left(s^{*,1}-\xi\right)\ge 2l_{0}\xi$, thus $\mathcal{M}\left(s+\xi\right)-l_{0}\xi\ge l_{0}\xi\textgreater 0$. This leads to our conclusion if we set $\widetilde{\mathcal{C}}\left(s\right)=4/\mathcal{M}\left(s^{*,1}-\xi\right)$ and provide $m \gtrsim \xi ^{-4}\cdot w^2\left(\text{cone}\left(\mathcal{K}\right)\cap \mathbb{S}^{n-1}\right)$.
	\end{proof}
	
	The demonstration of Theorem \ref{theorem3} primarily necessitates the construction of counterexamples, we now provide specific construction methods.	
		
	\begin{proof}[Proof of Theorem \ref{theorem3}]
		When $\pmb{\omega}=0$ and $\x\in \mathbb{R}^{n}$, we consider the loss function $$\mathcal{L}_1\left(\x\right)=\lV\lv\mA\x\rv-\pmb{b}\rV_{1}=\lV\lv\mA\x\rv-\lv\mA\x_{0}\rv-\pmb{z}\rV_{1}.$$
	Thus $\mathcal{L}_1\left(\x_{0}\right)=\lV\pmb{z}\rV_{1}$. 
			
			Firstly, we present the following lemma, which is a variant of Lemma \ref{A}.
			\begin{lemma}\label{C}
				Fix $\x,\y\in \mathbb{R}^{n}$ and let $\rho=\frac{\langle\x, \y\rangle}{\lV\x\rV_{2}\cdot\lV\y\rV_{2}}$, $\alpha=\lV\y\rV_{2}/\lV\x\rV_{2}$. 
				For $s \textgreater s^{*,1}$, there exists a positive numerical constant $c$ corresponding to $\tilde{l}$  such that with probability exceeding  $1-\mathcal{O}\left(e^{-cm\xi^{2}}\right)$, we have
				\begin{equation}
				\begin{split}
	\min_{\mS \subset[m],  |\mS|\leq s m}\lV\lv\mA_{{\mS}^{c}}\pmb{x}\rv-\lv\mA_{{\mS}^{c}}\pmb{y}\rv\rV_{1}&-\lV\lv\mA_{\mS}\pmb{x}\rv-\lv\mA_{\mS}\pmb{y}\rv\rV_{1}\\
	&\le m\left[M\left(\rho,\alpha, s-\xi\right)+\tilde{l}\xi/2\right]\cdot\mathrm{dist}_{1}\left(\x,\y\right).
					\end{split}
				\end{equation}
			\end{lemma}
Secondly, let $\left(\tilde{\rho}_{1},\tilde{\alpha}\right)=\argmin M\left(\rho,\alpha,s^{*,1}\right)$,
			that is $M\left(\tilde{\rho}_{1},\tilde{\alpha},s^{*,1}\right)=\mathcal{M}\left(s^{*,1}\right)=0$. Without loss of generality, let $\lV\x_{0}\rV_{2}=1$, then set
			 \begin{eqnarray}
			 \x_{\star}=\tilde{\rho}_{1}\tilde{\alpha}\x_{0}+\sqrt{1-\left(\tilde{\rho}_{1}\right)^{2}} \tilde{\alpha}\x_{0}^{\perp},
			 \end{eqnarray}
			 where $\lV\x_{0}^{\perp}\rv_{2}=1$. Thus, $\lV\x_{\star}\rV_{2}=\tilde{\alpha}$ and $\frac{\langle\x_{\star},\x_{0}\rangle}{\lV\x_{\star}\rV_{2}\cdot\lV\x_{0}\rV_{2}}=\tilde{\rho}_{1}$. 			
			
			Thridly, we construct $\pmb{z}$. Let $\mS_{0}$ be the support of the largest $sm$ absolute value of entries in $\lv\mA\pmb{x}_{\star}\rv-\lv\mA\pmb{x}_{0}\rv$.
			We establish the adversarial sparse outlier:
			\begin{eqnarray}
			\pmb{z} =  \lv\mA_{\mS_{0}}\pmb{x}_{\star}\rv-\lv\mA_{\mS_{0}}\pmb{x}_{0}\rv.
			\end{eqnarray}
			Thus, $\mathcal{L}_1\left(\x_{0}\right)=\lV\lv\mA_{\mS_{0}}\pmb{x}_{\star}\rv-\lv\mA_{\mS_{0}}\pmb{x}_{0}\rv\rV_{1}$ and 
			$\mathcal{L}_1\left(\x_{\star}\right)=\lV\lv\mA_{\mS_{0}^{c}}\pmb{x}_{\star}\rv-\lv\mA_{\mS_{0}^{c}}\pmb{x}_{0}\rv\rV_{1}$.
			
			Finally, as $s-2\xi\textgreater s^{1,*}$, there exists $\tilde{l}\textgreater0$ such that 
			$$M\left(\tilde{\rho}_{1},\tilde{\alpha}, s-\xi\right)\textless M\left(\tilde{\rho}_{1},\tilde{\alpha}, s^{*,1}+\xi\right)=-\tilde{l}\xi.$$ 
			By Lemma \ref{C} 
			 , we have
			\begin{equation*}
			\begin{split}
				\mathcal{L}_1\left(\x_{\star}\right)-\mathcal{L}_1\left(\x_{0}\right)&=
				\lV\lv\mA_{\mS_{0}^{c}}\x_{\star}\rv-\lv\mA_{\mS_{0}^{c}}\x_{0}\rv\rV_{1}-
				\lV\lv\mA_{\mS_{0}}\x_{\star}\rv-\lv\mA_{\mS_{0}}\x_{0}\rv\rV_{1}\\
				&=\min_{\mS \subset[m],  |\mS|\leq s m} 
				\lV\lv\mA_{{\mS}^{c}}\x_{\star}\rv-\lv\mA_{{\mS}^{c}}\x_{0}\rv\rV_{1}-
				\lV\lv\mA_{{\mS}}\x_{\star}\rv-\lv\mA_{{\mS}}\x_{0}\rv\rV_{1}\\
				&\le m\left[M\left(\tilde{\rho}_{1},\tilde{\alpha}, s^{*,1}+\xi\right)+\tilde{l}\xi/2\right]\cdot \mathrm{dist}_{1}\left(\x_{\star},\x_{0}\right) \\
				&=\frac{-\tilde{l}m\xi}{2}\cdot\mathrm{dist}_{1}\left(\x_{\star},\x_{0}\right)\textless0.
				\end{split}
			\end{equation*}	
			with probability exceeding  $1-\mathcal{O}\left(e^{-cm\xi^{2}}\right)$.	
		\end{proof}

\subsection{Recovery Guarantee for Intensity Measurement}\label{example}
		The proofs of Theorem \ref{theorem2} and Theorem \ref{theorem4} are similar to the amplitude case. 
	\begin{proof}[Proof of Theorem \ref{theorem2}]
	By Lemma \ref{ASOSC1} and Theorem \ref{k=2}, with probability exceeding $1-\mathcal{O}\left(e^{-cm\xi^{2}}\right)$,
	\begin{eqnarray*}
	\mathrm{dist}_{2}\left(\x_{\star},\x_{0}\right)\leq \frac{2}{\mathcal{J}(s+\xi)-l_{0}\xi}\frac{\lV\pmb{\omega}\rV_1}{m}.
	\end{eqnarray*}
	$\mathcal{J}\left(s\right)$ is monotonically decreasing, then we have $\mathcal{J}\left(s+\xi\right)\textgreater \mathcal{J}\left(s^{*,2}-\xi\right)\textgreater 0$. 
	The constant $l_{0}$ can be chosen small enough to make $ \mathcal{J}\left(s^{*,2}-\xi\right)\ge 2l_{0}\xi$, thus $\mathcal{J}\left(s+\xi\right)-l_{0}\xi\ge l_{0}\xi\textgreater 0$. 
	If $m \gtrsim \xi ^{-4}\cdot w^2\left(\text{cone}\left(\mathcal{K}\right)\cap \mathbb{S}^{n-1}\right)$ and $\widetilde{\mathcal{C}}\left(s\right)=4/\mathcal{J}\left(s^{*,2}-\xi\right)$, we then can finish the proof.
	\end{proof}
		
		\begin{proof}[Proof of Theorem \ref{theorem4}]
		Let $\mathcal{L}_2\left(\x\right)=\lV\lv\mA\x\rv-\lv\mA\x_{0}\rv^{2}-\pmb{z}\rV_{1}$.
			And set $\tilde{\rho}_{2}=\argmin J\left(\rho,s^{*,2}\right)$, thus $J\left(\tilde{\rho}_{2},s^{*,2}\right)=\mathcal{J}\left(s^{*,2}\right)=0$. 
			Then let $\x_{\star}$ be the vector that satisfies $\tilde{\rho}_{2}=\langle\frac{\x_{0}-\x_{\star}}{\lV\x_{0}-\x_{\star}\rV_{2}},\frac{\x_{0}+\x_{\star}}{\lV\x_{0}+\x_{\star}\rV_{2}}\rangle$. 
			Set $\mS_{0}$ be the support of the largest $sm$ absolute value of entries in $\lv\mA\pmb{x}_{\star}\rv^{2}-\lv\mA\pmb{x}_{0}\rv^{2}$.
			Then the adversarial sparse outlier we construct here is 
			\begin{eqnarray}
			\pmb{z} =  \lv\mA_{\mS_{0}}\pmb{x}_{\star}\rv^{2}-\lv\mA_{\mS_{0}}\pmb{x}_{0}\rv^{2}.
			\end{eqnarray}
				We next provide the following lemma, which is similar to Lemma \ref{B}. 
			\begin{lemma}\label{D}
				Fix $\x,\y\in \mathbb{R}^{n}$ and let $\rho=\langle\frac{\x-\y}{\lV\x-\y\rV_{2}},\frac{\x+\y}{\lV\x+\y\rV_{2}}\rangle$. 
				For $s \textgreater s^{*,2}$, there exists a positive numerical constant $c$ corresponding to $\tilde{l}$ such that with probability at least $1-\mathcal{O}\left(e^{-cm\xi^{2}}\right)$,	
				\begin{equation*}
				\begin{split}
					\min_{\mS \subset[m],  |\mS|\leq s m}\lV\lv\mA_{{\mS}^{c}}\pmb{x}\rv^{2}-\lv\mA_{{\mS}^{c}}\pmb{y}\rv^{2}\rV_{1}&-\lV\lv\mA_{\mS}\pmb{x}\rv^{2}-\lv\mA_{\mS}\pmb{y}\rv^{2}\rV_{1}\\
					&\leq m\left[J\left(\rho, s-\xi\right)+\tilde{l}\xi/2\right]\cdot\mathrm{dist}_{2}\left(\x,\y\right).
					\end{split}
				\end{equation*}
			\end{lemma}		
			Finally, there exists $\tilde{l} \textgreater0$ such that $J\left(\tilde{\rho}_{2},s-\xi\right)\textless J\left(\tilde{\rho}_{2},s^{*,2}+\xi\right)=-l\xi$, 
		thus by Lemma \ref{D} , we have
			\begin{equation*}
			\begin{split}
				\mathcal{L}_{2}\left(\x_{\star}\right)-\mathcal{L}_{2}\left(\x_{0}\right)&=
				 \lV\lv\mA_{\mS_{0}^{c}}\x_{\star}\rv^{2}-\lv\mA_{\mS_{0}^{c}}\x_{0}\rv^{2}\rV_{1}-
				\lV\lv\mA_{\mS_{0}}\x_{\star}\rv^{2}-\lv\mA_{\mS_{0}}\x_{0}\rv^{2}\rV_{1}\\
				&=\min_{\mS \subset[m],  |\mS|\leq s m} 
				\lV\lv\mA_{{\mS}^{c}}\x_{\star}\rv^{2}-\lv\mA_{{\mS}^{c}}\x_{0}\rv^{2}\rV_{1}-
				\lV\lv\mA_{{\mS}}\x_{\star}\rv^{2}-\lv\mA_{{\mS}}\x_{0}\rv^{2}\rV_{1}\\
				&\le m\left[J\left(\tilde{\rho}_{2}, s^{*,2}+\xi\right)+l\xi/2\right]\cdot\mathrm{dist}_{2}\left(\x_{\star},\x_{0}\right) \\
				&=\frac{-\tilde{l}m\xi}{2}\cdot\mathrm{dist}_{2}\left(\x_{\star},\x_{0}\right)\textless0.
				\end{split}
			\end{equation*}
			with probability exceeding $1-\mathcal{O}\left(e^{-cm\xi^{2}}\right)$.
		\end{proof}
		
\section {Numerical Experiments for Sharp Thresholds}\label{6}
In this section, we conduct experiments to study the fraction of sparse outliers that the two nonlinear LAD model can tolerate even when the sample size is very large. 
For simplicity, we do not take into account the dense noise, i.e., $\pmb{\omega} = \pmb{0}$. We let the ground-truth $\pmb{x}_0 \in \mathbb{R}^{5}$ be a standard Gaussian vector and the adversarial sparse outliers are selected following from the counterexamples in the proof of Theorem \ref{theorem3} and Theorem \ref{theorem4}. We choose the sample size $m = 500n$ and the entries of the sample vectors $\{\pmb{a}_i\}_i$ are i.i.d. drawn from $\mathcal{N}\left(0, \pmb{I}_{n}\right)$. We plot the relative error $\frac{\| \x-\x_0\|_2}{\|\x\|_2}$ with the fraction of sparse noise ranging from $0$ to $1$ in Figure \ref{Fig}. We can observe that it can tolerate no more than an $s^{*,1}\approx 0.2043$ fraction of corruptions in amplitude-based LAD and $s^{*,2}\approx 0.1185$ in intensity-based LAD. It also is clear that amplitude-based LAD is more robust than  intensity-based LAD.
These meet our results in the preceding article.

\begin{figure}[htbp]
\centering 

\begin{minipage}[b]{0.45\textwidth}
\centering 
\includegraphics[width=1\textwidth]{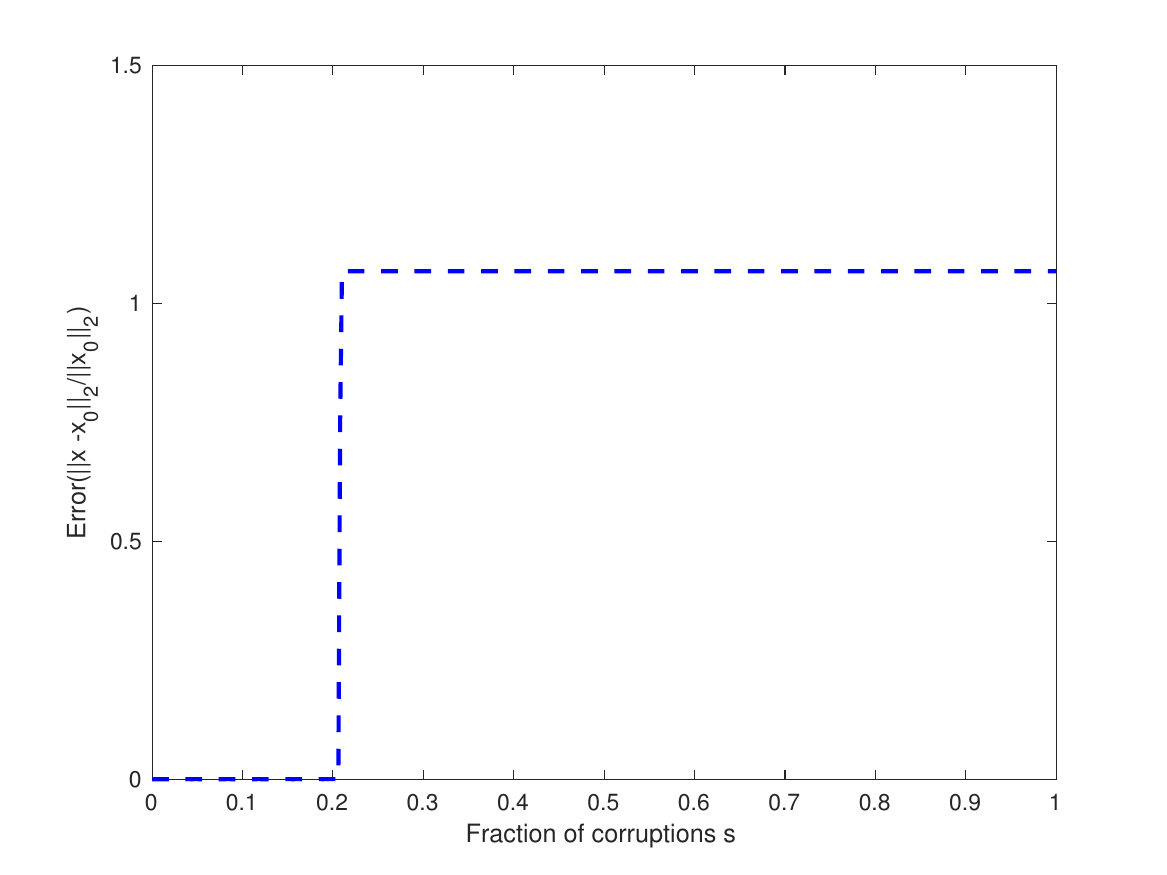} 
\subcaption*{(b) The recovery threshold for nonlinear LAD based on amplitude measurement.}
\label{Fig1}
\end{minipage}
\quad
\begin{minipage}[b]{0.45\textwidth} 
\centering 
\includegraphics[width=1\textwidth]{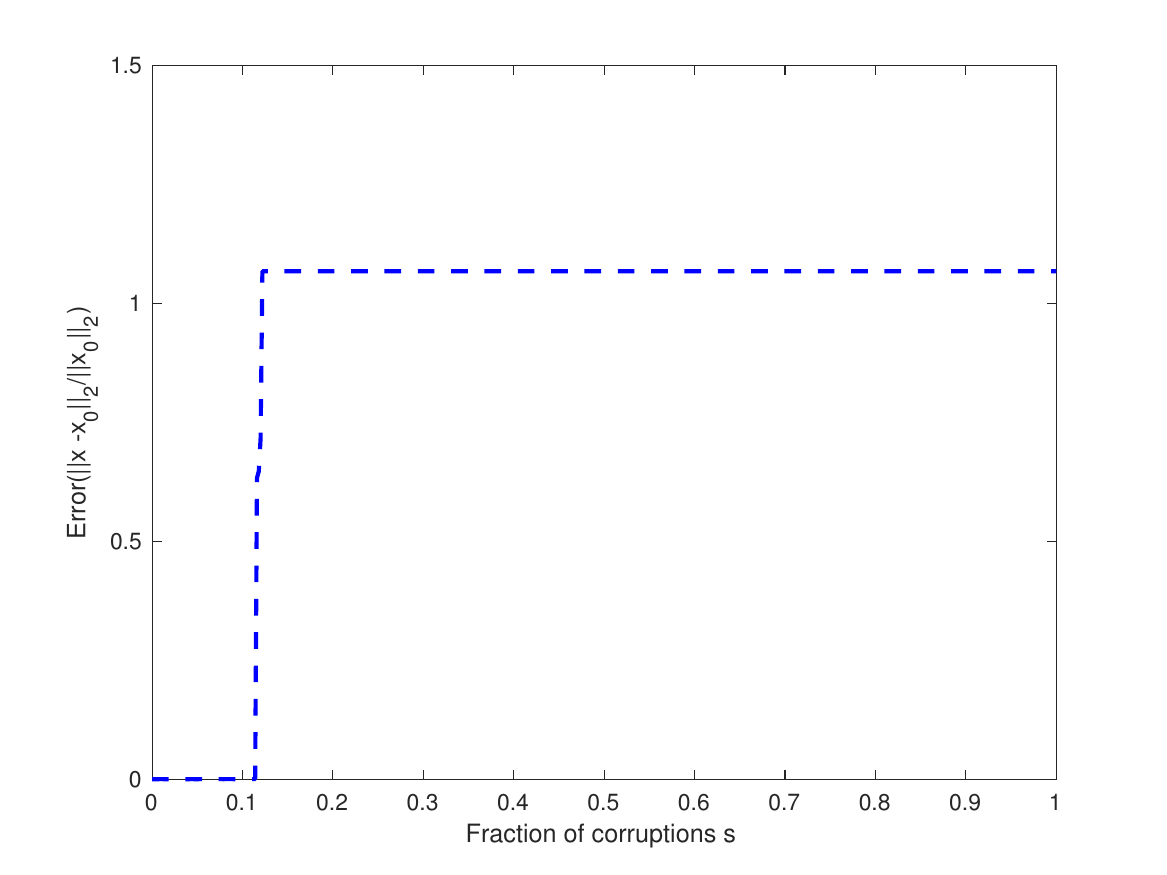}  
\subcaption*{(a) The recovery threshold for nonlinear LAD based on intensity measurement.}

\label{Fig2}
\end{minipage}
\caption{The threshold $s$ for amplitude-based and intensity-based nonlinear LAD model.}
\label{Fig}
\end{figure}

	\appendix
 \section*{Appendix}
	\section{Dvoretzky-Kiefer-Wolfowitz Type Inequality}	\label{conin}	
	The following lemma is a corollary of the standard Dvoretzky-Kiefer-Wolfowitz inequality, similar proof can be found in \cite{karmalkar2019compressed,xu2022low}.	
	\begin{lemma}\label{DKW}
		Let $\Gamma=\{ \xi_1, \dots, \xi_m\}$ be i.i.d. sampled from any continuous distribution function $F$ on $\mathbb{R}$. Then for any $\eta, \epsilon \in [0, 1]$, the following holds with probability at least $1 - 4e^{-2m\epsilon^2}$,
		\begin{eqnarray*}
			F^{-1}\left(\eta- \epsilon\right) < \widehat{F}^{-1}(\eta) < F^{-1}\left(\eta+ \epsilon\right).
		\end{eqnarray*}
		Here, $\widehat{F}$ is the associated empirical distribution function defined by $\widehat{F}\left(x\right) = {1\over m} \sum_{i=1}^m\mathbbm{1}_{\xi_i \leq x}$.
	\end{lemma} 
		\begin{proof}
	Standard Dvoretzky-Kiefer-Wolfowitz inequality \cite{massart1990tight} stated that
	\begin{eqnarray*}
			\mathbb{P}\left(\sqrt{m}\sup_{x}\lv\widehat{F}\left(x\right)-F\left(x\right)\rv\ge\lambda\right)\le2\exp\left(-2\lambda^{2}\right).
		\end{eqnarray*}
		Set $x=F^{-1}\left(\eta- \epsilon\right)$ and $\lambda=\epsilon\sqrt{m}$, then 
		\begin{eqnarray*}
			\mathbb{P}\left(\lv\widehat{F}\left(t\right)-\left(\eta-\epsilon\right)\rv\ge\epsilon\right)\le2\exp\left(-2m\epsilon^{2}\right).
		\end{eqnarray*}
		Then we have
		\begin{eqnarray*}
			\mathbb{P}\left(\widehat{F}\left(t\right)\le\eta\right)\ge1-2\exp\left(-2m\epsilon^{2}\right).
		\end{eqnarray*}
		Monotonicity of  $\widehat{F}\left(x\right) $ then implies the left-side inequality. The right-side inequality follows similarly.
	\end{proof}
	
\section{PDF of $\lv Z_{\rho}\rv=\lv X\cdot Y\rv$}\label{RV}
	We first consider the PDF of $Z_{\rho}=X\cdot Y$. 
		We define the parametric integral as follows:
		\begin{eqnarray*}
			I_{\alpha,\beta}\left(t\right)=\int_{-\infty}^{\infty}\frac{e^{{\rm i}xt}}{\sqrt{\left(x-{\rm i}\alpha\right)\left(x+{\rm i}\beta\right)}} dx,
		\end{eqnarray*}
		where ${\rm i}=\sqrt{-1}, a>0,b>0 $ and $-\infty <t<\infty$ are fixed real numbers. Setting $x=y+\frac{{\rm i}}{2}\left(\alpha-\beta\right)$ in $I_{\alpha,\beta}\left(t\right)$, we then obtain
		\begin{eqnarray*}
			I_{\alpha,\beta}\left(t\right)&=&e^{\frac{\beta-\alpha}{2}t}\int_{-\infty}^{\infty}\frac{e^{{\rm i}ty}}{\sqrt{y^{2}+(\frac{\beta+\alpha}{2})^{2}}}dy\\
			&=&2e^{\frac{\beta-\alpha}{2}t}\int_{0}^{\infty}\cos\left(ty\right)\left[y^{2}+(\frac{\beta+\alpha}{2})^{2}\right]^{-1/2}dy
			=2e^{\frac{\beta-\alpha}{2}t}K_{0}\left(\frac{\beta+\alpha}{2}\lv t\rv\right),\\
		\end{eqnarray*}
		where the  last step stems from the definition of Bessel function of the second kind of order zero $K_{0}\left(x\right)$. Meanwhile, from \cite{craig1936frequency}, the characteristic function of $Z_{\rho}=X\cdot Y$ is 
		\begin{eqnarray*}
			\mathbb{E}e^{{\rm i}tZ}=\frac{1}{\sqrt{\left(1-{\rm i}\left(1+\rho\right)t\right)\left(1+{\rm i}\left(1+\rho\right)t\right)}}.
		\end{eqnarray*}
		Hence, by the inversion theorem, the PDF of $Z_{\rho}$ can be expressed as
		\begin{eqnarray*}\label{rv2}
			f_{Z_{\rho}}\left(z\right)
			&=&\frac{1}{2\pi}\int_{-\infty}^{\infty}\frac{e^{-{\rm i}tz}}{\sqrt{\left(1-{\rm i}\left(1+\rho\right)t\right)\left(1+{\rm i}\left(1+\rho\right)t\right)}}dt\\
			&=&\frac{1}{2\pi\sqrt{1-\rho^{2}}}\int_{-\infty}^{\infty}\frac{e^{-{\rm i}tz}}{\sqrt{\left(t-\frac{1}{{\rm i}\left(1+\rho\right)}\right)\left(t+\frac{1}{{\rm i}\left(1+\rho\right)}\right)}}dt
			=\frac{1}{\pi\sqrt{1-\rho^2}}e^{ \frac{\rho z}{1-\rho^2}}K_0\left( \frac{\lv z\rv}{1-\rho^2} \right).
		\end{eqnarray*}
	Actually when $\rho \to \pm 1$, the variable $Z_{\rho}= X\cdot Y$ converges in distribution to chi-square random variable $\chi^2\left(1\right)$ or $-\chi^2\left(1\right)$. 	
	Finally, due to $f_{\rho}\left(z\right)=f_{Z_{\rho}}\left(z\right)+f_{Z_{\rho}}\left(-z\right)$ we finish the proof.
		
\section{Partial Proofs of Nonlinear ROB Condition}\label{partial}
\subsection{Proof of Proposition \ref{J}}
		\noindent(a)
		The expectation of $\lv Z_{\rho}\rv$ derived from \cite{huang2022outlier} is
		\begin{eqnarray*}
			\mathbb{E}\left[\lv Z_{\rho}\rv\right]=\frac{2}{\pi}\left[\sqrt{1-\rho^{2}}+\rho\arcsin(\rho)\right].
		\end{eqnarray*}
		By definition, $\mathcal{J}\left(0\right)=\min\limits_{\rho \in[0,1]}\mathbb{E}\left[\lv Z_{\rho}\rv \right]$ and $\mathcal{J}\left(1\right)=-\max\limits_{\rho \in[0,1]}\mathbb{E}\left[\lv Z_{\rho}\rv\right]$.
		Let
		\begin{eqnarray*} 
			\phi \left(\rho\right)=\sqrt{1-\rho^{2}}+\rho\arcsin\left(\rho\right).
		\end{eqnarray*}
		Then, by calculation $ \frac{\partial \phi}{\partial\rho} =  \arcsin\left(\rho\right) \geq 0.$
		Thus we can get 
		\begin{eqnarray*} 
			\mathcal{J}\left(0\right)=\frac{2}{\pi}\phi \left(0\right)=\frac{2}{\pi}\quad\text{and}\quad\mathcal{J}\left(1\right)=-\frac{2}{\pi}\phi \left(1\right)=-1.
		\end{eqnarray*}
		The continuity of $\mathcal{J}\left(s\right)$ is obvious. For monotonicity, let $0\le s_{1} \textless s_{2} \le1$ and $\rho_{1} = \arg\min \mathcal{J}\left(s_{1}\right)$, $\rho_{2} = \arg\min \mathcal{J}\left(s_{2}\right)$.
		By definition, 
		\begin{eqnarray*} 
			\mathcal{J}\left(s_{1}\right)=J\left(\rho_{1},s_{1}\right)\textgreater J\left(\rho_{1},s_{2}\right)\ge J\left(\rho_{2},s_{2}\right)=\mathcal{J}\left(s_{2}\right).
		\end{eqnarray*}
		As $\mathcal{J}\left(0\right)$ and $ \mathcal{J}\left(1\right)$ are finite, $\mathcal{J}\left(s\right)$ is a well-defined function.
	
	\noindent(b) 
	The uniqueness of zero can be derived by Proposition \ref{J}.(a). 
	The calculation of $s^{*,2}$ has been given in \cite{huang2022outlier}. 
	Though our minimum balance function $\mathcal{J}\left(s\right)$ is different from that in \cite{huang2022outlier}, the difference is only the coefficient of $\rho$, which does not affect the value of $s^{*,2}$.

\subsection{Proof of Lemma \ref{B}}		
		Let $h_{1}\left(x\right) = x \cdot \mathbbm{1}_{[0,t]}\left(x\right)$, and $\Gamma \left(t\right) = h_{1}\left(\lvert Z_{\rho}\rvert\right)$. 
		Similar to {\bf Step 2} in Theorem \ref{k=1}, we can derive that for all $\rho\in\left[0,1\right]$ and $t\in \mathbb{R}^{+}\cup \{+\infty\}$, 
		there exists $ K_{0}>0$ such that $	\lV\Gamma\left(t\right) \rV_{\psi_{1}} \leq K_{0}$. Let $r\textgreater s\in (0,1)$ be a fixed constant and $t = F_{\rho}^{-1}\left(1-r\right)$. 
		We consider the sampling set $\Gamma=\{\Gamma_{1}\left(t\right),\cdots,\Gamma_{m}\left(t\right)\}$ and get
		\begin{eqnarray*}
			\mathbb{E}\left[\Gamma \left(t\right)\right]=\int_{0 }^{t}
			zf_{\rho }\left(z\right)dz=J_{1}\left(\rho,r\right).
		\end{eqnarray*}	
		Thus by Bernstein inequality in \cite{vershynin2018high}, we have
		\begin{eqnarray}\label{X1}
			\mathbb{P}\left(\lv {1\over m} \sum_{i=1}^m \Gamma_i\left(t\right) - \mathbb{E}\left(\Gamma \left[t\right)\right] \rv >\varepsilon_1\right) \leq 2e^{-c_{0}m\min\left\{\varepsilon_1^2/K_{0}^{2},\varepsilon_1/K_{0}\right\}}.
		\end{eqnarray} 		
		Based on Dvoretzky-Kiefer-Wolfowitz type inequality in Lemma \ref{DKW} and similar to the argument in {\bf Step 2} of Theorem \ref{k=1}, we can get
		\begin{eqnarray}\label{X2}
			\min_{\mS \subset[m],  |\mS|\leq s m}\frac{\lV\lv\mA_{{\mS}^{c}}\pmb{x}\rv^{2}-\lv\mA_{{\mS}^{c}}\pmb{y}\rv^{2}\rV_{1}}{\mathrm{dist}_{2}\left(\x,\y\right)}
			\ge \sum_{i=1}^m \Gamma_i\left(t\right)
		\end{eqnarray} 
		with probability  exceeding $1 - 4e^{-2m\left(r-s\right)^2}$. 
		Combining (\ref{X1}) with (\ref{X2}), we get
		\begin{eqnarray*}
			\frac{1}{m}\min_{\mS \subset[m],  |\mS|\leq s m}\frac{\lV\lv\mA_{{\mS}^{c}}\pmb{x}\rv^{2}-\lv\mA_{{\mS}^{c}}\pmb{y}\rv^{2}\rV_{1}}{\mathrm{dist}_{2}\left(\x,\y\right)}
			\ge \frac{1}{m}\sum_{i=1}^m  \Gamma_i\left(t\right) \ge \mathbb{E}\left[\Gamma \left(t\right)\right] - \varepsilon_1 = J_{1}\left(\rho,s+\epsilon\right)-\varepsilon_1
		\end{eqnarray*}
		with certain probability. 
		By setting $\epsilon =\xi$ and $\varepsilon_1= l_{1}\xi$ for small enough $l_{1}$ such that $\varepsilon_1^{2}/K_{0}^{2}\le\varepsilon_1/K_{0}$, we finally get (\ref{B5}) with probability  at least $1 -\mathcal{O}(e^{-c_{1}m\xi^{2}})$.	
		Similarly, we can establish (\ref{B6}) if we 
		set $h_{2}\left(x\right) = x\cdot \mathbbm{1}_{\left[t,+\infty\right)}\left(x\right)$ and $\Gamma \left(t\right)=h_{2}\left(\lv Z_{\rho}\rv\right)$.

	\normalem
	\bibliographystyle{plain}
	\bibliography{ref}
	
\end{document}